\documentclass[11pt,reqno,british,a4paper,twoside]{amsart}

\usepackage[sc]{mathpazo}
\usepackage{mathtools,url,amssymb,enumitem}
\usepackage[main=british]{babel}
\usepackage{csquotes,simplewick,microtype}

\mathtoolsset{mathic=true}
\setlist[enumerate,1]{label=(\roman*)}

\providecommand{\texorpdfstring}[2]{#1}

\makeatletter
\g@addto@macro\bfseries{\boldmath}
\makeatother

\theoremstyle{plain}
\newtheorem{theorem}{Theorem}[section]
\newtheorem{proposition}[theorem]{Proposition}
\newtheorem{lemma}[theorem]{Lemma}
\newtheorem{corollary}[theorem]{Corollary}

\theoremstyle{definition}
\newtheorem{definition}[theorem]{Definition}
\newtheorem{remark}[theorem]{Remark}
\newtheorem{example}[theorem]{Example}

\numberwithin{equation}{section}

\newcommand{\Lie}[1]{\operatorname{#1}}
\newcommand{\lie}[1]{\operatorname{\mathfrak{#1}}}
\newcommand{\G}{\Lie{G}}
\newcommand{\GL}{\Lie{GL}}
\newcommand{\SL}{\Lie{SL}}
\newcommand{\SO}{\Lie{SO}}
\newcommand{\Sp}{\Lie{Sp}}
\newcommand{\SU}{\Lie{SU}}

\newcommand{\so}{\lie{so}}
\newcommand{\su}{\lie{su}}
\newcommand{\lt}{\lie{t}}

\newcommand{\Hodge}{\mathop{}\!{*}}
\newcommand{\hook}{\mathbin{\lrcorner}}

\newcommand{\bC}{\mathbb{C}}
\newcommand{\bH}{\mathbb{H}}
\newcommand{\bR}{\mathbb{R}}

\newcommand{\cU}{\mathcal{U}}

\newcommand{\tW}{\widetilde{W}}

\DeclareMathOperator{\tr}{tr}
\DeclareMathOperator{\re}{Re}
\DeclareMathOperator{\im}{Im}
\DeclareMathOperator{\diag}{diag}
\DeclareMathOperator{\rank}{rank}
\DeclareMathOperator{\Stab}{Stab}
\DeclareMathOperator{\vol}{vol}
\DeclareMathOperator{\Hol}{Hol}
\DeclareMathOperator{\adj}{adj}

\DeclareFontFamily{U}{bigeuf}{}
\DeclareFontShape{U}{bigeuf}{m}{n}{<-6>s*[1.5]eufm5
<6-8>s*[1.5]eufm7
<8->s*[1.5]eufm10}{}
\DeclareSymbolFont{bigeufletters}{U}{bigeuf}{m}{n}
\DeclareMathSymbol{\sumcic}{\mathop}{bigeufletters}{`S}

\DeclarePairedDelimiter{\Span}{\langle}{\rangle}
\DeclarePairedDelimiter{\norm}{\lVert}{\rVert}
\DeclarePairedDelimiter{\abs}{\lvert}{\rvert}
\DeclarePairedDelimiter{\halfopen}{\lbrack}{\rparen}
\DeclarePairedDelimiter{\halfclosed}{\lparen}{\rbrack}

\newcommand{\with}{\mid}
\newcommand{\SetSymbol}[1][]{\nonscript\:#1|
  \allowbreak\nonscript\:\mathopen{}}
\DeclarePairedDelimiterX{\Set}[1]{\lbrace}{\rbrace}{%
  \renewcommand{\with}{\SetSymbol[\delimsize]}%
#1 }

\DeclarePairedDelimiterX{\inp}[2]{\langle}{\rangle}{#1,#2}

\newcommand{\D}[2]{\frac{\partial #1}{\partial #2}}
\newcommand{\Dsq}[2]{\frac{\partial^2 #1}{\partial #2^2}}
\newcommand{\Dsqm}[3]{\frac{\partial^2 #1}{\partial #2\partial #3}}

\newcommand{\any}{\,\cdot\,}

\newcommand{\eqbreak}[1][2]{\\&\hspace{#1em}}
\newcommand{\eqand}[1][1]{\hspace{#1em}\text{and}\hspace{#1em}}
\newcommand{\eqcond}[2][2]{\hspace{#1em}\text{#2}}

\newcommand{\hyphen}{\nobreak-\nobreak\hskip0pt}

\begin{document}

\title{Toric geometry of \( \G_2 \)-manifolds}

\author[T.~B. Madsen]{Thomas Bruun Madsen}

\address[T.~B. Madsen]{School of Computing, University of Buckingham\\
Hunter Street, Buckingham\\
MK18 1EG\\
United Kingdom\\
and
Centre for Quantum\\
Geometry of Moduli Spaces\\
Aarhus University\\
Ny Munkegade 118, Bldg 1530\\
8000 Aarhus\\
Denmark}

\email{thomas.madsen@buckingham.ac.uk}

\author[A.~F. Swann]{Andrew Swann}

\address[A.~F. Swann]{Department of Mathematics, Centre for Quantum
Geometry of Moduli Spaces, and Aarhus University Centre for
Digitalisation, Big Data and Data Analytics\\
Aarhus University\\
Ny Munkegade 118, Bldg 1530\\
8000 Aarhus\\
Denmark}

\email{swann@math.au.dk}

\begin{abstract}
  We consider \( \G_2 \)-manifolds with an effective torus action that
  is multi-Hamiltonian for one or more of the defining forms.
  The case of \( T^3 \)-actions is found to be distinguished.
  For such actions multi-Hamiltonian with respect to both the three-
  and four-form, we derive a Gibbons-Hawking type ansatz giving the
  geometry on an open dense set in terms a symmetric
  \( 3\times 3 \)-matrix of functions.
  This leads to particularly simple examples of explicit metrics with
  holonomy equal to~\( \G_2 \).
  We prove that the multi-moment maps exhibit the full orbit space
  topologically as a smooth four-manifold containing a trivalent graph
  as the image of the set of special orbits and describe these graphs
  in some complete examples.
\end{abstract}


\maketitle

\begin{center}
  \begin{minipage}{0.8\linewidth}
    \microtypesetup{protrusion=false} \small \tableofcontents
  \end{minipage}
\end{center}

\section{Introduction}
\label{sec:intro}

The Gibbons-Hawking ansatz \cite{Gibbons-Hawking1} furnishes a way of
constructing hyperK\"ahler four-manifolds with circle symmetry.
More generally, the classifications of complete hypertoric manifolds
(see, e.g., \cite{Bielawski:tri-Hamiltonian,Dancer-S:hK}) show that
moment map techniques, similar to the Delzant construction of
symplectic geometry, can be useful when exploring Ricci-flat metrics.

Metrics of holonomy \( \G_2 \) are known to be Ricci-flat.
What is perhaps less familiar is that also in this setting, one has a
notion of (multi-)symplectic geometry
\cite{Madsen-S:mmmap1,Madsen-S:mmmap2}.
It is therefore natural to ask what should be the analogue of toric or
hypertoric geometry in this context.

The first question to consider is which tori can act in a
multi-Hamiltonian way on a torsion-free \( \G_2 \)-manifold.
We find in~\S\ref{sec:eff-action} that the torus must have rank
between \( 2 \) and~\( 4 \).
A dimension count then reveals that the case that best mimics
hypertoric geometry is when a three-torus is multi-Hamiltonian for
both the defining three-form and its Hodge dual four-form: this is the
only case where the dimension of the orbit space matches the dimension
of the target space for the multi-moment map.
This `toric' case with an effective \( T^3 \)-action enjoys several
immediate properties in common with the standard toric and hypertoric
situation.
In particular, we see that all stabilisers are again connected
subtori, in this case of dimension at most~\( 2 \), and that the
multi-moment maps provide local coordinates on the manifold of
principal orbits, so an open dense set of~\( M \) becomes a
\( T^3 \)-bundle over a four-manifold.

In~\S\ref{sec:G2can-form}, we derive the analogue of the
Gibbons-Hawking ansatz for toric \( \G_2 \)-manifolds~\( M \).
The crucial local datum is now a smooth positive definite section
\( V\in\Gamma(U,S^2(\bR^3) ) \) on an open set in
\( U \subset \bR^4 \).
This determines the curvature of the \( T^3 \)-bundle and must satisfy
a pair of PDEs: one is a divergence-free condition on~\( V \) and the
other system is a quasi-linear elliptic second order PDE\@.
These differential operators are natural for the action
of~\( \GL(3,\bR) \) resulting from change of basis for the Lie algebra
\( \lie t^3 \) of~\( T^3 \), and are nearly uniquely specified by this
property.
The divergence-free equation is essentially one used in continuum
mechanics.

The above description, in terms of \( V \), applies at points that
have trivial \( T^3 \)-stabiliser.
In~\S\ref{sec:sing-orb}, we obtain a good understanding of the
differential topology near singular orbits.
As in the hypertoric case, one finds that \( M/T^3 \) is homeomorphic
to a smooth manifold.
This is unlike the situation for toric symplectic manifolds where the
orbit space is a manifold with corners~\cite{Karshon-L:toric}.
Our main result is that such a homeomorphism is realised via the
multi-moment maps.
Furthermore, the image of the singular orbits in the four-manifold
\( M/T^3 \) is a trivalent graph, whose edges are straight lines in
multi-moment map coordinates.
These results are obtained by first studying flat models, including
\( S^1 \times \bC^3 \), where the graph has a single vertex where
three edges meet, and \( T^2 \times \bR \times \bC^2 \), where the
graph has one edge and no vertex.

Our distinguished case of \( \G_2 \)-manifolds that are
multi-Hamiltonian for \( T^3 \) has the good feature that there are
non-trivial complete examples with full holonomy~\( \G_2 \).
Indeed, the Bryant-Salamon \( \G_2 \)-structure on the spin bundle of
\( S^3 \) \cite{Bryant-S:excep-hol} is such an example, as are the
generalisations
in~\cite{Brandhuber-al:G2,Bazaikin-B:G2,Bogoyavlenskaya:G2}.
We study the Bryant-Salamon example in some detail, showing how it
fits into the general framework.
In particular, the associated trivalent graph is connected with two
vertices and the multi-moment map provides a global homeomorphism
\( M/T^3 \to \bR^4 \).

If one is willing to compromise on completeness, our approach produces
particularly simple Riemannian metrics with (restricted) holonomy
equal to \( \G_2 \), see Examples~\ref{ex:mu-dep}
and~\ref{ex:poly-ex}.

\subsection*{Acknowledgements}

We thank Uwe Semmelmann for useful discussions.
TBM is grateful for financial support by Villum Fonden.
AFS was partially supported by the Danish Council for Independent
Research~\textbar~Natural Sciences project DFF - 6108-00358 and the
Danish National Research Foundation grant DNRF95 (Centre for Quantum
Geometry of Moduli Spaces - QGM).
We thank the referee for a careful reading of the paper, insightful
comments and suggesting better forms for \( z_{j}^{i} \) and \( Q \)
in~\S\ref{sec:torsion-free}.

\section{\texorpdfstring{\( \G_2 \)}{G2}-manifolds with
multi-Hamiltonian torus actions}
\label{sec:eff-action}

Let \( M \) be a connected \( 7 \)-manifold.
A \( \G_2 \)-structure on \( M \) is determined by a \( 3 \)-form
\( \varphi \) that is pointwise linearly equivalent to the form
\begin{equation*}
  \varphi_0 =
  e^{123}-e^1(e^{45}+e^{67})-e^2(e^{46}+e^{75})-e^3(e^{47}+e^{56}),
\end{equation*}
where \( E_1,\dots, E_7 \) is a basis of \( V\cong \bR^7 \),
\( e^1,\dots, e^7 \) is its dual basis of \( V^* \), wedge signs are
suppressed and \( e^{123} = e^{1}\wedge e^{2} \wedge e^{3} \), etc.
We shall sometimes refer to \( E_1,\dots E_7 \) (and its dual) as an
\emph{adapted basis}.

The \( \GL(V) \)-stabiliser of \( \varphi_0 \) is the compact
\( 14 \)-dimensional Lie group \( \G_2\subset\SO(V) \).
In fact, \( \varphi_0 \) uniquely determines both the inner product
\( g_0 = \sum_{j=1}^7(e^j)^2 \) and volume element
\( \vol_0 = e^{1234567} \) via the relation
\begin{equation*}
  6g_0(X,Y)\vol_0
  = (X\hook \varphi_0)\wedge(Y\hook\varphi_0)\wedge\varphi_0,
\end{equation*}
for all \( X,Y\in V \) (see~\cite{Bryant:excep-hol}).
Correspondingly, \( \varphi \)~determines a metric~\( g \) and a
volume form~\( \vol \) on~\( M \).
From this, it also follows that we have an additional dual
\( 4 \)-form, \( \Hodge\varphi \), pointwise equivalent to
\begin{equation*}
  \Hodge\varphi_0
  = e^{4567}-e^{23}(e^{45}+e^{67})-e^{31}(e^{46}+e^{75})-e^{12}(e^{47}+e^{56}).
\end{equation*}
We also get a cross-product operation via
\( g(X\times Y,Z) = \varphi(X,Y,Z) \).
Three-dimensional subspaces of~\( T_p M \) closed under the
cross-product are \emph{associative}, their orthogonal complements are
\emph{co-associative}.

Following standard terminology, we say that \( (M,\varphi) \) is a
\emph{\( \G_2 \)-manifold} if the \( \G_2 \)-structure is
torsion-free, hence the (restricted) holonomy group \( \Hol_0(g) \) is
contained in \( \G_2\subset\SO(7) \).
This implies \( g \) is Ricci-flat.
It is well-known~\cite{Fernandez-G:G2-manf} that being torsion-free,
in this context, is equivalent to the condition that \( \varphi \) is
closed and co-closed.

We are interested in \( \G_2 \)-manifolds that come with an effective
action of a torus~\( T^k \) on~\( M \) that preserves~\( \varphi \),
hence also \( \Hodge\varphi \) and the metric \( g \).
Such an action gives us a map
\begin{equation}
  \label{eq:inf-act}
  \xi\colon \bR^k\cong\lt^k\to\mathfrak{X}(M),
\end{equation}
which is a Lie algebra anti-homomorphism.
Subsequently, we shall often write \( \xi_p \) for the image of
\( \xi \) at \( p \in M \).
This is a subspace of \( T_p M \) of dimension at most~\( k \).

\begin{definition}[{\cite[Def.~3.5]{Madsen-S:mmmap2}}]
  Let \( N \) be a manifold equipped with a closed \( (k+1)
  \)-form~\( \alpha \), and \( G \)~an Abelian Lie group acting
  on~\( N \) preserving~\( \alpha \).
  A \emph{multi-moment map} for this action is an invariant map
  \( \nu\colon N \to \Lambda^k\lie g^* \) such that
  \begin{equation*}
    d\inp{\nu}{W}=\xi(W)\hook\alpha,
  \end{equation*}
  for all \( W\in\Lambda^k\lie g \); here
  \( \xi(W)\in\Gamma(\Lambda^k TM) \) is the unique multi-vector
  determined by~\( W \) via~\( \xi \).
\end{definition}

We say that such a torus symmetry on a \( \G_2 \)-manifold is
\emph{multi-Hamiltonian} if there are multi-moment maps associated
with \( (\varphi,T^k) \) and/or \( (\Hodge\varphi,T^k) \).
This requires that \( k\geqslant 2 \) for non-triviality.
A discussion of circle invariant \( \G_2 \)-metrics can be found
in~\cite{Apostolov-S:G2S1}, and such metrics were also at the heart of
the constructions in~\cite{Foscolo-al:G2S1}.

Given an effective torus action by \( T^k \) on \( (M,\varphi) \), it
is obvious that \( k\leqslant7 \) as we have the following well-known
observation:

\begin{lemma}
  \label{lem:symm-degree}
  Let \( N \) be an \( n \)-manifold with an effective action of a
  torus~\( T^k \).
  Then \( k\leqslant n \) and the principal stabiliser is trivial.
\end{lemma}

\begin{proof}
  It suffices to prove the final statement.
  As \( T^k \) is Abelian, conjugation is trivial.
  Therefore different isotropy subgroups~\( H_p \) belong to different
  isotropy types.
  It follows that the principal stabiliser can be obtained as the
  intersection of all stabilisers, \( \bigcap_{p\in N} H_p \), and so
  is the trivial group by effectiveness of the action.
\end{proof}

If \( N \) is a compact Ricci-flat manifold, then each Killing vector
field is parallel~\cite{Bochner:Killing}.
It follows by~\cite[Cor.~6.67]{Besse:Einstein} that \( (N,h) \) has a
finite cover in the form of a Riemannian product
\( T^\ell\times N_1^{n-\ell} \), some
\( k\leqslant \ell\leqslant n \), of a flat torus and compact
simply-connected Ricci-flat manifold~\( N_1 \).
In particular, for a compact \( \G_2 \)-manifold with an effective
\( T^k \)-action, \( \Hol_0(g) \)~is a proper subgroup of~\( \G_2 \).
From Berger's classification~\cite{Beger:hol-clsf}, it follows that
the restricted holonomy is trivial, \( \SU(2) \) or~\( \SU(3) \).
Correspondingly, we must have \( \ell=7 \), \( \ell=3 \) or
\( \ell=1 \), respectively.

As our main interest is the case of full holonomy, we will often
concentrate on the case when \( M \) is non-compact.

Focusing on multi-Hamiltonian actions, we have already established
that our torus must have rank between \( 2 \) and \( 7 \).
It turns out there are further restrictions.

\begin{proposition}
  \label{prop:rank-multh-tor}
  If \( T^k \) acts effectively on a \( \G_2 \)-manifold and is
  multi-Hamiltonian, then \( 2\leqslant k\leqslant4 \).
\end{proposition}

The proof of Proposition~\ref{prop:rank-multh-tor} is an immediate
consequence of Lemmas~\ref{lem:ass-coass} and~\ref{lem:mult-hamil}
below.

\begin{lemma}
  \label{lem:ass-coass}
  Suppose \( W \) is a \( 5 \)-dimensional subspace of
  \( (V,\varphi_0) \).
  Then \( W \) contains both associative and co-associative subspaces.
\end{lemma}

\begin{proof}
  Choose an orthonormal basis \( E_1,E_2 \) for \( W^\perp \).
  Then \( E_3 = E_1\times E_2 \) lies in~\( W \).
  Thus \( W \) contains the co-associative subspace
  \( \Span{E_1,E_2,E_3}^\perp \).
  Furthermore, \( E_1,E_2,E_3 \) can be extended to a \( \G_2 \)
  adapted basis for \( V \).
  For this basis \( E_4\times E_7 = E_3 \), so
  \( \Span{E_3,E_4,E_7} \) is an associative subspace of~\( W \).
\end{proof}

The following observation states that a necessary condition for an
action to be multi-Hamiltonian is that the orbits are
\enquote{isotropic}.

\begin{lemma}
  \label{lem:mult-hamil}
  If a torus action of~\( T^k \) on~\( N \) is multi-Hamiltonian for a
  closed differential form~\( \alpha \) of degree~\( r\leqslant k \),
  then \( \alpha|_{\Lambda^r\xi}\equiv0 \).

  If \( b_1(N)=0 \), this condition is also sufficient for the
  \( T^k \)-action to be multi-Hamiltonian.
\end{lemma}

\begin{proof}
  Consider the fundamental vector fields
  \( \xi(V_1),\dots,\xi(V_{r-1}) \) associated with vectors
  \( V_1,\dots V_{r-1}\in\lie{t}^k\cong\bR^k \), and let \( \nu' \) be
  a component of the multi-moment map
  \( \nu\colon N \to \Lambda^{r-1}\lt^k \) that satisfies
  \begin{equation*}
    d\nu'=\alpha(\xi(V_1),\dots,\xi(V_{r-1}),\any).
  \end{equation*}

  By invariance of the multi-moment map, we have for any
  \( V_r\in \lt^k \) that
  \begin{equation*}
    0
    = \mathcal L_{\xi(V_r)}\nu'
    = \xi(V_r)\hook d\nu'
    = \alpha(\xi(V_1),\dots,\xi(V_r)).
  \end{equation*}
  It follows that \( \alpha \) vanishes on \( \Lambda^r\xi \), as
  required.

  As \( T^k \) preserves \( \alpha \), the \( 1 \)-form
  \( \alpha(\xi(V_1),\dots,\xi(V_{r-1}),\any) \) is closed and
  therefore exact, say equal to \( d\nu' \), when \( b_1(N) = 0 \).
  The condition \( \alpha|_{\Lambda^r\xi}\equiv0 \) implies invariance
  of~\( \nu' \), since \( T^k \) is connected.
\end{proof}

The upshot of Proposition~\ref{prop:rank-multh-tor} is that there are
potentially \( 7 \) possible cases that can occur:
\( T^2 \)~multi-Hamiltonian for \( \varphi \),
\( T^3 \)~multi-Hamiltonian for either \( \varphi \) or
\( \Hodge\varphi \), \( T^3 \)~multi-Hamiltonian for both
\( \varphi \) and \( \Hodge\varphi \), \( T^4 \)~acts
multi-Hamiltonian for \( \varphi \) or \( \Hodge\varphi \), and
\( T^4 \) acts multi-Hamiltonian for both \( \varphi \) and
\( \Hodge\varphi \).
In reality, the last situation cannot occur as we shall explain below.

Let \( M_0\subset M \) denote subset of points \( p \) such that the
map~\( \xi \) of~\eqref{eq:inf-act} is injective.
It follows by Lemma~\ref{lem:symm-degree} that \( M_0 \) is open and
dense, since it contains the set of principal orbits~\( M'_0 \).
Note that \( M'_0 \) is the total space of a principle
\( T^k \)-bundle.

\subsection{Two-torus actions}
\label{sec:T2sym}

This case was studied in~\cite{Madsen-S:mmmap1}, so we shall only give
a brief summary.

Given a multi-Hamiltonian action for~\( \varphi \), the multi-moment
map \( \nu \) is an invariant scalar function
\( M \to \Lambda^2(\lt^2)^*\cong\bR \).
For \( t \in \nu(M) \), if the action of \( T^2 \)~is free on the
level set \( \nu^{-1}(t) \), then the reduction
\( N = \nu^{-1}(t)/T^2 \) is a \( 4 \)-manifold carrying three
symplectic forms of the same orientation, induced by
\begin{equation*}
  U_1 \hook \varphi,\quad
  U_2 \hook \varphi \eqand
  U_1 \wedge U_2\hook\Hodge\varphi,
\end{equation*}
where \( U_i \) generate the \( T^2 \)-action.
In interesting cases this triple is not hyperK\"ahler, but does fit in
to the framework of~\cite{Fine-Y:hypersymplectic}.

Conversely, the \( \G_2 \)-manifold \( (M,\varphi) \) can be recovered
from the \( 4 \)-manifold \( N \) by building a two-torus bundle over
it.
One then equips the total space of this bundle with a suitable
\( \SU(3) \)-structure and reconstructs the original
\( \G_2 \)-holonomy manifold via an adapted \enquote{Hitchin flow}.

Known complete \( \G_2 \)-manifolds with a multi-Hamiltonian
\( T^2 \)-action include the Bryant-Salamon metrics on the space of
anti-self-dual \( 2 \)-forms over a complete self-dual positive
Einstein manifold~\cite{Bryant-S:excep-hol}.

\subsection{Three-torus actions}
\label{sec:eff-T3}

The main interest here will be for actions that are multi-Hamiltonian
for both \( \varphi \) and \( \Hodge\varphi \), so that we have
multi-moment maps \( (\nu,\mu)\colon M \to \bR^3\times\bR \).
This is the only case in which the dimension of \( M/T^k \) matches
that of the target space for the multi-moment maps.
Being multi-Hamiltonian for \( \varphi \), it follows by
Lemma~\ref{lem:mult-hamil} that \( \varphi|_{\Lambda^3\xi}\equiv0 \).
This condition was studied in~\cite[\S IV]{Harvey-Lawson:calibrated},
where it where it is shown that \( \G_2 \) acts transitively on the
set of such three-planes.
Indeed, for \( p\in M_0 \), for any orthonormal
\( X_2,X_3 \in \xi_p \), there is an adapted basis where these
correspond to \( E_6 \) and~\( E_7 \).
The \( \G_2 \)-stabiliser of \( \Set{E_6,E_7} \) is an
\( \SU(2) \)~acting on \( \Span{E_2,E_3,E_4,E_5} \cong \bC^2 \).
Using this action, we see that we can extend to a basis
\( X_1,X_2,X_3 \) of~\( \xi_p \) and have \( X_1 \) identified
with~\( E_5 \).
Now \( \hat\theta_i = e^{i+4} \), \( i=1,2,3 \), are dual to
\( X_1,X_2,X_3 \): \( \hat\theta_i(X_j)=\delta_{ij} \) and
\( \hat\theta_i(X)=0 \) for \( X\perp\Span{X_1,X_2,X_3} \).  Putting
\begin{equation*}
  \begin{gathered}
    \alpha_i=X_j\wedge X_k\hook\varphi=-e^i,\quad\beta=X_1\wedge X_2
    \wedge X_3\hook\Hodge\varphi=-e^4,
  \end{gathered}
\end{equation*}
where \( (ijk)=(123) \), corresponding to the differentials of the
multi-moment maps at~\( p \), the \( \G_2 \)-structure at
\( p\in M_0 \) takes the form:
\begin{equation}
  \label{eq:G2formsa}
  \begin{gathered}
    \varphi = - \alpha_{123} -
    \alpha_1(\beta\hat{\theta}_1-\hat{\theta}_{23}) -
    \alpha_2(\beta\hat\theta_2-\hat\theta_{31})
    - \alpha_3(\beta\hat\theta_3-\hat\theta_{12}),\\
    \Hodge\varphi = \hat\theta_{123}\beta +
    \alpha_{23}(\beta\hat{\theta}_1-\hat{\theta}_{23}) +
    \alpha_{31}(\beta\hat\theta_2-\hat\theta_{31}) +
    \alpha_{12}(\beta\hat\theta_3-\hat\theta_{12}).
  \end{gathered}
\end{equation}
We shall return to this expression later on,
in~\S\ref{sec:G2can-form}, refining it to give a \( \G_2 \)-analogue of
the Gibbons-Hawking ansatz.

As in the hypertoric case, there are no points with discrete
stabiliser.
In particular, \( M_0 \) is the total space of a principal
\( T^3 \)-bundle over the corresponding orbit space.

\begin{lemma}
  \label{lem:conn-isotr}
  Suppose \( T^3 \) acts effectively on a manifold \( M \) with
  \( \G_2 \)-structure \( \varphi \) so that the orbits are isotropic,
  \( \varphi|_{\Lambda^3\xi_p} = 0 \).
  Then each isotropy group is connected and of dimension at most two;
  hence trivial, a circle or~\( T^2 \).
\end{lemma}

\begin{proof}
  Let \( p \in M \) have isotropy group \( H \leqslant T^3 \).
  Then \( H \) is an Abelian group acting on \( V = T^\bot \), where
  \( T = T_p(T^3\cdot p) \) is the tangent space to the orbit.
  As \( T^3\cdot p \) has an neighbourhood that can be identified with
  the normal bundle~\( T^3\times_H V \) and this neighbourhood
  necessarily intersects principal orbits, the action on~\( V \) is
  faithful.
  Adding the trivial \( H \)-module \( T \) to \( V \), we have that
  the \( H \)-action on \( T_p M = T\oplus V \) preserves the
  \( \G_2 \)-structure.
  As \( \G_2 \) has rank~\( 2 \), we get \( \dim H \leqslant 2 \).

  If \( \dim H = 0 \), then at \( p \), then any generators
  \( U_1,U_2,U_3 \) of the \( T^3 \) have the property that their
  cross products span~\( T_p M \).
  As the \( T^3 \)-action preserves the \( \G_2 \)-structure, this
  implies that \( H \) fixes every element of \( T_p M \).
  Thus \( H \) is trivial.

  For \( \dim H = 1 \), the space \( T \) is spanned by two linearly
  independent vectors \( U_1 \) and~\( U_2 \).
  It follows that \( H \) preserves the non-zero vector
  \( U_1 \times U_2 \) in~\( V \) and must act as a subgroup
  of~\( \SU(2) \) on the orthogonal complement.
  Thus \( H \) is a one-dimensional Abelian subgroup of~\( \SU(2) \).
  This forces the identity component \( H_0 \) to be a maximal torus
  of \( \SU(2) \), so conjugate to
  \( T^1 = \Set{\diag(\exp(i\theta),\exp(-i\theta)) \with \theta \in
  \bR} \).
  But any matrix in \( \SU(2) \) commuting with \( T^1 \) is diagonal,
  so belongs to \( T^1 \).  Thus \( H \cong T^1 \), which is connected.

  When \( \dim H = 2 \), then \( H \) is a subgroup of~\( \SU(3) \),
  so its identity component is a maximal torus.
  Again this conjugate to a group of diagonal matrices
  \( \diag(\exp(i\theta),\allowbreak \exp(i\varphi),\allowbreak
  \exp{-i(\theta+\varphi)}) \) and any other matrix commuting with
  this group is of this form.
  Thus \( H \cong T^2 \) and is connected.
\end{proof}

The classical example of a complete \( \G_2 \)-holonomy manifold with
a multi-Hamiltonian \( T^3 \)-action is the spin bundle of~\( S^3 \)
equipped with its Bryant-Salamon structure~\cite{Bryant-S:excep-hol},
see~\S\ref{sec:Cone-and-BS-S3S3}.
Additional complete examples can be found
in~\cite{Brandhuber-al:G2,Bazaikin-B:G2,Bogoyavlenskaya:G2}.

\subsection{Four-torus actions}
\label{sec:T4sym}

If a torus \( T^4 \) is multi-Hamiltonian for \( \varphi \), then the
multi-moment map has \( 6 \) components as its image is in
\( \Lambda^2(\lt^4)^*\cong\bR^6 \).

\begin{lemma}
  Suppose \( (M,\varphi) \) admits an effective \( T^4 \)-action that
  is multi\hyphen Hamiltonian for \( \varphi \).
  If \( p\in M_0 \), then \( \xi_p\subset T_p M \) is co-associative.
\end{lemma}

\begin{proof}
  Take a pair \( E_1,E_2 \) of orthonormal vectors in~\( \xi_p \).
  As \( \varphi|_{\Lambda^3\xi}\equiv0 \), we have that
  \( E_3=E_1\times E_2 \) lies in \( \xi_p^\perp \).
  We may extend \( E_1,E_2,E_3 \) to an adapted basis
  \( E_1,\dots, E_7 \).
  Using the stabiliser \( \SU(2) \) of \( E_1,E_2 \) in~\( \G_2 \), we
  may ensure that \( E_4\in\xi_p \).
  Then the relations \( E_1\times E_4=E_5 \) and
  \( E_2\times E_4=E_6 \) give \( \xi_p^\perp=\Span{E_3,E_5,E_6} \),
  and so \( \xi_p=\Span{E_1,E_2,E_4,E_7} \).
  In particular, \( \xi_p^\perp \) is associative and \( \xi_p \) is
  co-associative.
\end{proof}

A local description of \( \G_2 \)-manifolds with \( T^4 \)-symmetry
whose orbits are co-associative is given
in~\cite{Baraglia:co-associative}, and also discussed
in~\cite{Donaldson:fibrations}.
Essentially these correspond to positive minimal immersions into
\( \bR^{3,3}\cong H^2(T^4) \), and this in turn is the image of the
multi-moment map.

If \( T^4 \) is multi-Hamiltonian for \( \Hodge\varphi \), we get a
multi-moment map with \( 4 \) components as it has values in
\( \Lambda^3\lt^4\cong\bR^4 \).

\begin{lemma}
  Suppose \( T^4 \) acts effectively on \( (M,\varphi) \) and is
  multi-Hamiltonian for \( \Hodge\varphi \).
  If \( p\in M_0 \), then the \( 4 \)-dimensional subspace
  \( \xi_p\leqslant T_p M \) contains an associative subspace.
  In particular, the action can not be multi-Hamiltonian
  for~\( \varphi \).
\end{lemma}

\begin{proof}
  Choose a pair of orthonormal vectors \( E_1,E_2\in\xi_p \) and
  extend these to an adapted basis for~\( T_p M \).
  As before, we may now use the stabiliser \( \SU(2)\leqslant \G_2 \)
  of \( E_1,E_2 \) to ensure that \( E_4\in \xi_p \).
  Now \( \Hodge\varphi|_{\Lambda^4\xi}\equiv0 \) implies that
  \( E_7=E_1\times E_2\times E_4 \) lies in \( \xi_p^\perp \).
  Therefore, \( \xi_p=\Span{E_1,E_2,E_4,v} \) with \( v \) a unit
  vector in \( \Span{E_3,E_5,E_6} \).

  As \( \Span{E_3,E_4,E_7} \) is associative, there is a circle
  subgroup of \( \G_2 \) that acts via multiplication by \( e^{it} \)
  on \( \bC^2\cong\Span{E_1+iE_2,E_5+iE_6} \).
  Using this, we may ensure that \( v\in\Span{E_3,E_5} \).
  Writing \( v=xE_3+yE_5 \), we find that \( E_1\times v=-xE_2+yE_4 \)
  so that \( \xi_p \) contains the associative subspace
  \( \Span{E_1,xE_2-yE_4,xE_3+yE_5} \).
\end{proof}

All currently known examples of complete \( \G_2 \)-manifolds with a
multi-Hamiltonian action of \( T^4 \) have reduced holonomy.

\section{Toric \texorpdfstring{\( \G_2 \)}{G2}-manifolds: local
characterisation}
\label{sec:G2can-form}

Motivated by the discussion in~\S\ref{sec:eff-action}, we introduce
the following terminology:

\begin{definition}
  \label{def:toric}
  A \emph{toric \( \G_2 \)-manifold} is a torsion-free
  \( \G_2 \)-manifold \( (M,\varphi) \) with an effective action of
  \( T^3 \) multi-Hamiltonian for both \( \varphi \)
  and~\( \Hodge\varphi \).
\end{definition}

The purpose of this section is to derive an analogue of the
Gibbons-Hawking ansatz~\cite{Gibbons-Hawking1,Gibbons-Hawking2} for
toric \( \G_2 \)-manifolds, more specifically obtaining a local form
for a toric \( \G_2 \)-structure and describing the torsion-free
condition in these terms.
An independent derivation of such equations with an extension to
\( \SU(2) \)-actions was obtained by~\cite{Chihara:G2-from-SL-fibr}
after our announcement~\cite{Swann:PADGE}.

So assume \( (M,\varphi) \) is a toric \( \G_2 \)-manifold, with
\( T^3 \) acting effectively.
Let \( U_1,U_2,U_3 \) be infinitesimal generators for the
\( T^3 \)-action, then these give a basis for
\( \xi_p\leqslant T_p M \) for each~\( p\in M_0 \).
Denote by \( \theta = (\theta_1,\theta_2,\theta_3)^t \) the dual basis
of \( \xi_p^*\leqslant T_p^*M \):
\begin{equation*}
  \theta_i(U_j)=\delta_{ij} \eqand
  \theta(X)=0 \eqcond[1]{for all \( X\perp U_1,U_2,U_3 \).}
\end{equation*}
For brevity we write \( \theta_{ab} \) for
\( \theta_a \wedge \theta_b \), etc.

Let \( \nu = (\nu_1,\nu_2,\nu_3)^t \) and \( \mu \) be the associated
multi-moment maps; these satisfy
\begin{equation*}
  \begin{split}
    d\nu_i&=U_j\wedge U_k\hook\varphi=(U_j\times U_k)^\flat,
            \qquad(ijk)=(123), \\
    d\mu&=U_1\wedge U_2\wedge
          U_3\hook\Hodge\varphi.
  \end{split}
\end{equation*}
It follows from~\S\ref{sec:eff-T3} that \( (d\nu,d\mu) \) has full
rank on~\( M_0 \) and induces a local diffeomorphism
\( M_0/T^3\to\bR^4 \).
We define a \( 3\times3 \)-matrix \( B \) of inner products given by
\begin{equation*}
  B_{ij}=g(U_i,U_j),
\end{equation*}
and on~\( M_0 \) we put \( V=B^{-1}=\det(B)^{-1}\adj(B) \).

In these terms, we have the following local expression for the
\( \G_2 \)-structure:

\begin{proposition}
  \label{prop:G2-can-form}
  On \( M_0 \), the \( 3 \)-form \( \varphi \) and \( 4 \)-form
  \( \Hodge\varphi \) are
  \begin{gather*}
    \varphi = -\det(V)d\nu_{123} + d\mu \wedge d\nu^t\adj(V)\theta +
    \sumcic_{ijk}\theta_{ij}\wedge d\nu_k,
    \\
    \Hodge\varphi = \theta_{123}d\mu +
    \tfrac1{2\det(V)}\bigl(d\nu^t\adj(V)\theta\bigr)^2 +
    \det(V)d\mu\wedge\sumcic_{ijk}\theta_i\wedge d\nu_{jk}.
  \end{gather*}
  The associated \( \G_2 \)-metric is given by
  \begin{equation}
    \label{eq:G2-metr}
    g = \tfrac1{\det V}\theta^t\adj(V)\theta + d\nu^t\adj(V)d\nu +
    \det(V)d\mu^2.
  \end{equation}
\end{proposition}

We note that \( M_0 \) comes with a co-associative foliation with
\( T^3 \)-symmetry whose leaves are specified by setting \( \nu \)
equal to a constant.
The corresponding distribution is given by the kernel of
\( d\nu_{123} \).
In particular, the restriction of \( \Hodge\varphi \) to the
each leaf is \(\theta_{123}d\mu \).

\begin{proof}
  We start by choosing an auxiliary symmetric matrix \( A>0 \) such
  that \( A^2=B^{-1} \) which is possible as \( B \) is positive
  definite.
  Then we set \( X_i=\sum_{j=1}^3A_{ij}U_j \) and observe that
  \begin{equation*}
    g(X_i,X_j)=(ABA)_{ij}=(A^2B)_{ij}=\delta_{ij},
  \end{equation*}
  showing that the triplet \( (X_1,X_2,X_3) \) is orthonormal.
  It follows that we can apply the formulae~\eqref{eq:G2formsa} for
  \( \varphi \) and~\( \Hodge\varphi \).

  We make the identification \( \bR^3\cong \Lambda^2\bR^3 \) via
  contraction with the standard volume form.
  Then if we let \( \Lambda^2A \) denote the induced action of \( A \)
  on \( \Lambda^2\bR^3 \), we can get
  \begin{equation*}
    \Lambda^2A = \det(A)A^{-1}.
  \end{equation*}
  In these terms, we have that
  \begin{equation*}
    \alpha=(\Lambda^2A)d\nu,\quad
    \beta=\det(A)d\mu \eqand
    \hat\theta=A^{-1}\theta=\tfrac1{\det(A)}(\Lambda^2A)\theta.
  \end{equation*}

  Turning to the expressions for the \( \G_2 \) three-form, we start
  by noting that
  \begin{equation*}
    \alpha_{123}=\det(\Lambda^2A)d\nu_{123}
  \end{equation*}
  and that \( \alpha_q(\beta\hat{\theta}_q-\hat{\theta}_{rs}) \)
  equals
  \begin{equation*}
    \sum_{i=1}^3(\Lambda^2A)_{qi}d\nu_i
    \biggl(\sum_{j=1}^3(\Lambda^2A)_{qj}d\mu\theta_j
    - \det(B)\sum_{k,\ell=1}^3(\Lambda^2A)_{rk}(\Lambda^2A)_{s\ell}
    \theta_{k\ell}\biggr),
  \end{equation*}
  where \( (qrs)=(123) \).  Summing these terms gives
  \begin{multline*}
    \varphi = -\det(\Lambda^2A)d\nu_{123} +
    d\mu\sum_{i,j=1}^3d\nu_i(\Lambda^2A)^2_{ij}\theta_j\\
    + \det(B)\sum_{i,k,\ell=1}^3
    (\Lambda^2A)_{1i}(\Lambda^2A)_{2k}(\Lambda^2A)_{3\ell}
    (d\nu_i\theta_{k\ell}+d\nu_k\theta_{\ell i}+d\nu_\ell\theta_{ik}),
  \end{multline*}
  which is simplified by observing that the expression in the second
  line above reduces to give
  \( d\nu_1\theta_{23}+d\nu_2\theta_{31}+d\nu_2\theta_{12} \), as
  required by the multi-moment map relations.
  The asserted expression for \( \varphi \) therefore follows by
  noting that \( (\Lambda^2A)^2=B/\det(B)=\adj(V) \).

  To rephrase the \( 4 \)-form expression, we observe that
  \begin{equation*}
    \hat\theta_{123}\beta=\theta_{123}d\mu,
  \end{equation*}
  consistent with the multi-moment map condition, and that
  \( \alpha_{rs}(\beta\hat\theta_q-\hat\theta_{rs}) \) equals
  \begin{equation*}
    \sum_{i,j=1}^3(\Lambda^2A)_{ri}(\Lambda^2A)_{sj}d\nu_{ij}
    \Bigl( \sum_{k=1}^3(\Lambda^2A)_{qk}d\mu\theta_k -
    \tfrac1{\det(A)^2}\sum_{k,\ell=1}^3
    (\Lambda^2A)_{rk}(\Lambda^2A)_{s\ell}\theta_{k\ell} \Bigr)
  \end{equation*}
  for \( (qrs)=(123) \).
  Upon summation, this quickly gives the stated expression for
  \( \Hodge\varphi \).

  Finally, for the metric we have
  \begin{equation*}
    \begin{split}
      g
      &= \hat\theta^t\hat\theta+\alpha^t\alpha+\beta^2
        = (A^{-1}\theta)^t A^{-1}\theta
        + (\Lambda^2Ad\nu)^t\Lambda^2Ad\nu+\det(A)^2d\mu^2\\
      &=\theta^t\bigl(\tfrac1{\det(V)}\adj(V)\bigr)\theta
        + d\nu^t\adj(V)d\nu + \det(V)d\mu^2,
    \end{split}
  \end{equation*}
  as claimed.
\end{proof}

\begin{remark}
  The expression for \( \Hodge\varphi \) may also be written as
  \begin{equation}
    \label{eq:Hphi2}
    \Hodge\varphi = \theta_{123}d\mu
    - \sumcic_{ijk}\sumcic_{pqr} V_{ip}d\nu_{jk}\theta_{qr} +
    \det(V)d\mu\wedge\sumcic_{ijk}\theta_i\wedge d\nu_{jk}.
  \end{equation}
\end{remark}

\begin{remark}
  \label{rem:GL3act-quadform}
  In the above, we have a natural action of \( \GL(3,\bR) \),
  corresponding to changing the basis of \( \lt^3 \).
  This action can sometimes be used to simplify arguments as it allows
  us to assume \( V \) is diagonal or the identity matrix at a given
  point provided only the \( \bR^3 = \widetilde{T^3} \) action is of
  relevance.
\end{remark}

\subsection{The torsion-free condition}
\label{sec:torsion-free}

Whilst it is true that any toric \( \G_2 \)-manifold can be expressed
as in Proposition~\ref{prop:G2-can-form}, the \( \G_2 \)-structure
captured by these formulae is not automatically torsion-free.

Computing \( d\varphi \) and \( d\Hodge\varphi \) involves the
exterior derivatives of \( \theta \).
By our observations in~\S\ref{sec:eff-T3}, we may think of
\( \theta \) as a connection \( 1 \)-form and its exterior derivative
\begin{equation*}
  d\theta=\omega=(\omega_1,\omega_2,\omega_3)^t
\end{equation*}
is therefore a curvature \( 2 \)-form (and as such represents an
integral cohomology class).
In terms of our parameterisation for the base space, via multi-moment
maps, we can write the curvature components of \( \omega \) in the
form
\begin{equation*}
  \omega_\ell=\sumcic_{ijk}(z_\ell^i d\nu_i d\mu + w_\ell^i d\nu_{jk}).
\end{equation*}
For convenience, we collect these curvature coefficients in two
\( 3\times 3 \) matrices \( Z=(z^i_j) \) and \( W=(w^i_j) \).

Closedness of \( \varphi \) now becomes:
\begin{equation}
  \label{eq:closedness}
  \begin{split}
    0
    &=-d\det(V)\wedge
      d\nu_{123} + d\mu(d\nu)^t\adj(V)\omega + d\mu(d\nu)^t
      d(\adj(V))\wedge\theta
      \eqbreak + \sumcic_{ijk}(\omega_i d\nu_j-\omega_j d\nu_i)\theta_k.
  \end{split}
\end{equation}
More explicitly, by wedging with \( d\nu_i \), these equations
completely determine the \( 9 \) curvature functions \( z_i^j \):
\begin{equation}
  \label{eq:Zcurv}
  z_{i}^{\ell} = \D{\adj(V)_{k\ell}}{\nu_{j}} - \D{\adj(V)_{j\ell}}{\nu_{k}}
\end{equation}
where \( (ijk)=(123) \).
Note, in particular, that the above expressions imply that \( Z \) is
traceless, \( \tr(Z)=0 \).

In addition, upon wedging with \( d\mu \), we see that
equation~\eqref{eq:closedness} forces \( W \) to be symmetric,
\( w^i_j=w_i^j \).
Finally, it follows by wedging~\eqref{eq:closedness}
with~\( \theta_{123} \) that
\begin{equation}
  \label{eq:logc-mu-der}
  \inp[\Big]{\adj(V)}{\D{V}{\mu}-W}=0 ,
\end{equation}
where \( \inp{\any}{\any} \) is the standard inner product on
\( M_3(\bR)\cong\bR^{9} \).

Addressing co-closedness of \( \varphi \), we use \eqref{eq:Hphi2} to get
\begin{equation}
  \label{eq:co-closed}
  \begin{split}
    0&=d\Hodge\varphi \\
    &=\sumcic_{ijk}\omega_i\theta_{jk}d\mu
      - \sumcic_{ijk}\sumcic_{pqr} dV_{ip} \wedge d\nu_{jk} \theta_{qr}
      - \sumcic_{ijk}\sumcic_{pqr}V_{ip} d\nu_{jk}
      (\omega_{q}\theta_{r} - \theta_{q}\omega_{r})
      \eqbreak
      +d(\det(V))\wedge d\mu\sumcic_{ijk}\theta_i d\nu_{jk}.
  \end{split}
\end{equation}

The curvature functions~\( w^i_j \) are computed from the wedge
product of~\eqref{eq:co-closed}
with~\( d\nu_i\theta_j \) to be
\begin{equation}
  \label{eq:Wcurv}
  w^j_i=\D{V_{ij}}{\mu}
\end{equation}
and it follows that equation~\eqref{eq:logc-mu-der} automatically
holds.
If instead we wedge~\eqref{eq:co-closed} with \( d\mu\theta_i \) we find
that
\begin{equation}
  \label{eq:div-free}
  \sum_{i=1}^3 \D{V_{ij}}{\nu_i} = 0\qquad j=1,2,3.
\end{equation}
We shall occasionally refer to this first order underdetermined
elliptic PDE system as the \enquote{divergence-free} condition.
Coincidentally, \eqref{eq:div-free}~appears in the study of (linear)
elasticity in continuum mechanics, expressing that the stress tensor
is divergence-free (see, e.g.,
\cite{Eastwood:elasticity1,Eastwood:elasticity2}).  This equation
together with the expression for \( \adj V \)
allows us to rewrite the coefficients \( z_{j}^{i} \) as
\begin{equation}
  \label{eq:z-exp}
  z_{\ell}^{i} = \sum_{a=1}^{3} \D{V_{j\ell}}{\nu_{a}} V_{ka} -
  \D{V_{k\ell}}{\nu_{a}} V_{ja}\qquad (ijk) = (123).
\end{equation}
One may now check that there are no further relations
from~\eqref{eq:closedness} or~\eqref{eq:co-closed}.

There are only \( 6 \) additional equations, arising from the
condition \( d\omega=0 \).
Using~\eqref{eq:Wcurv}, \eqref{eq:div-free} and~\eqref{eq:z-exp},
these equations can be expressed in the form of a second order
non-linear elliptic PDE without zeroth order terms:
\begin{equation}
  \label{eq:elliptic}
  L(V) + Q(dV) = 0.
\end{equation}
Here the operator \( L \) is given by
\begin{equation*}
  L=\Dsq{}{\mu}+\sum_{i,j}V_{ij}\Dsqm{}{\nu_i}{\nu_j},
\end{equation*}
and so has the same principal symbol as the Laplacian for the metric
\( d\mu^2 + d\nu^t B d\nu \), which, up to a conformal factor of
\( \det(V) \), is the same as the restriction of the
\( \G_2 \)-metric~\eqref{eq:G2-metr} to the horizontal space.
The operator \( Q \) is the quadratic form in \( dV \) given
explicitly by
\begin{equation*}
  Q(dV)_{ij} = - \sum_{a,b=1}^{3} \D{V_{ia}}{\nu_{b}}\D{V_{jb}}{\nu_{a}}.
\end{equation*}

In summary, we have that the torsion-free condition determines \( Z \)
and \( W \) together with three first order equations and six second
order equations.
We therefore have the following local description of toric
\( \G_2 \)-manifolds.

\begin{theorem}
  \label{thm:toric-G2-charac}
  Any toric \( \G_2 \)-manifold can be expressed in the form of
  Proposition~\ref{prop:G2-can-form} on the open dense subset of
  principal orbits for the \( T^3 \)-action.

  Conversely, given a principal \( T^3 \)-bundle over an open subset
  \(\cU\subset \bR^4 \), parameterised by \( (\nu,\mu) \), together
  with \( V\in\Gamma(\cU,S^2(\bR^3)) \) that is positive definite at
  each point.
  Then the total space comes equipped with a \( \G_2 \)-structure of
  the form given in Proposition~\ref{prop:G2-can-form}.
  This structure is torsion-free, hence toric, if and only if the
  curvature matrices \( Z \) and~\( W \) are determined by~\( V \) via
  \eqref{eq:Zcurv} and~\eqref{eq:Wcurv}, respectively, and
  \( V \)~satisfies the divergence-free condition~\eqref{eq:div-free}
  together with the non-linear second order elliptic
  system~\eqref{eq:elliptic}.  \qed
\end{theorem}

Using this characterisation, it is not difficult to construct many
explicit incomplete examples of toric \( \G_2 \)-manifolds (see
\S\ref{sec:cases}).

As one would expect, solutions with \( V \) constant are trivial in
the following sense:

\begin{corollary}
  \label{cor:V-const}
  A toric \( \G_2 \)-manifold with \( V \) constant is flat and hence
  locally isometric to \( \bR^7 \).
\end{corollary}

\begin{proof}
  If \( V \) is constant, we may assume \( V\equiv1\).
  Now \( \det(V)=1 \) everywhere and therefore \( M_0=M \).
  Consequently, by Proposition~\ref{prop:G2-can-form}, we have a
  global orthonormal co-frame \( e^1,\dots, e^7 \) satisfying
  \( de^i=0 \) for all \( 1\leqslant i \leqslant 7 \).
\end{proof}

Let us conclude this section by remarking that \eqref{eq:div-free} can
be integrated to obtain what in a sense may be seen as an analogue of
the local potential for hypertoric manifolds
(cf.~\cite{Bielawski:tri-Hamiltonian}).
The following observation is also known from continuum mechanics.

\begin{proposition}
  \label{prop:pot-div-eqn}
  Assume that \( V\in\Gamma(\cU,S^2(\bR^3)) \)
  satisfies~\eqref{eq:div-free}, with \( \cU\subset\bR^3 \)
  simply-connected.
  Then there exists \( A\in\Gamma(\cU,S^2(\bR^3)) \) such that
  \begin{equation}
    \label{eq:Vpot-eq}
    \begin{gathered}
      V_{ii}=\Dsq{A_{jj}}{\nu_k} + \Dsq{A_{kk}}{\nu_j} -
      2\Dsqm{A_{jk}}{\nu_j}{\nu_k},\quad V_{ij}=
      \Dsqm{A_{ik}}{\nu_j}{\nu_k} + \Dsqm{A_{jk}}{\nu_k}{\nu_i} -
      \Dsq{A_{ij}}{\nu_k}-\Dsqm{A_{kk}}{\nu_i}{\nu_j},
    \end{gathered}
  \end{equation}
  where \( (ijk)=(123) \).
\end{proposition}

\begin{proof}
  We begin by noting that equation~\eqref{eq:div-free} can be written
  more concisely as \( d \Hodge_3 (Vd\nu) = 0 \), where
  \( \nu = (\nu_1,\nu_2,\nu_3)^t \) and \( \Hodge_3 \) is the flat
  Hodge star operator with respect to \( \nu \).
  It follows that \( \Hodge_3 Vd\nu \) is exact, i.e.,
  \( Vd\nu = \Hodge_3 d(W d\nu) \) for some
  \( W \in \Gamma(\cU,M_3(\bR)) \).  The symmetry of \( V \) is then
  \begin{equation*}
    \D{W_{iq}}{\nu_p} - \D{W_{ip}}{\nu_q} = \D{W_{js}}{\nu_r} -
    \D{W_{jr}}{\nu_s} \qquad (j p q) = (1 2 3) = (i r s).
  \end{equation*}
  For \( i=j \) this relation is trivial.
  For \( i\ne j \), order \( i \) and \( j \) and take \( k \) such
  that \( (ijk) = (123) \).
  Then \( p=k \), \( q=i \), \( r=j \), \( s=k \), so the symmetry is
  \begin{equation*}
    - \D{(W_{ii}+W_{jj})}{\nu_k} + \D{W_{ik}}{\nu_i} + \D{W_{jk}}{\nu_j}
    = 0.
  \end{equation*}
  This is the same as
  \begin{equation*}
    d\Hodge_3 (\tW d\nu) = 0,
  \end{equation*}
  where \( \tW = W^T - (\tr W) 1_3 \), which is a divergence-free
  condition.
  Thus \( \tW d\nu = \Hodge_3d(A d\nu) \), for some
  \( A\in\Gamma(\cU,M_3(\bR)) \).
  It follows that \( A \) determines the symmetric matrix \( V \).
  In detail, we have
  \( \tW_{ij} = \partial A_{iq}/\partial\nu_p - \partial
  A_{ip}/\partial \nu_q \), \( (j p q) = (123) \), so using
  \( W = \tW^T - \tfrac12 (\tr\tW) 1_3 \), we get
  \begin{equation*}
    \begin{split}
      V_{ij} = \D{W_{iq}}{\nu_p} - \D{W_{ip}}{\nu_q}
      &=
        \D{}{\nu_p}\Bigl(\D{A_{qs}}{\nu_r} - \D{A_{qr}}{\nu_s} -
        \tfrac12\delta_{iq}\sum_{t=1}^3(\D{A_{tv}}{\nu_u} -
        \D{A_{tu}}{\nu_v})\Bigr)
        \eqbreak
        - \D{}{\nu_q}\Bigl(\D{A_{ps}}{\nu_r} - \D{A_{pr}}{\nu_s} -
        \tfrac12\delta_{ip}\sum_{t=1}^3(\D{A_{tv}}{\nu_u} -
        \D{A_{tu}}{\nu_v})\Bigr),
    \end{split}
  \end{equation*}
  \( (j p q) = (1 2 3) = (i r s) = (t u v) \).
  To simplify this, consider separately the cases where \( i=j \) and
  where \( i\ne j \).
  First for \( i = j \), we get \( p = r \), \( q = s \) distinct
  from~\( i \), so
  \begin{equation*}
    V_{ii} = \Dsq{A_{rr}}{\nu_s} + \Dsq{A_{ss}}{\nu_r} -
    \Dsqm{(A_{sr}+A_{rs})}{\nu_r}{\nu_s}.
  \end{equation*}
  For \( i\ne j \), again rearrange and introduce \( k \) so that
  \( (i j k) = (1 2 3) \).
  Then \( p=k \), \( q=i \), \( r=j \), \( s=k \), and
  \begin{equation*}
    \begin{split}
      V_{ij} &= \D{}{\nu_k}\Bigl(\D{A_{ik}}{\nu_j} - \D{A_{ij}}{\nu_k} -
               \tfrac12\delta_{ii}\sum_{t=1}^3(\D{A_{tv}}{\nu_u} -
               \D{A_{tu}}{\nu_v})\Bigr) \eqbreak -
               \D{}{\nu_i}\Bigl(\D{A_{kk}}{\nu_j} - \D{A_{kj}}{\nu_k} -
               \tfrac12\delta_{ik}\sum_{t=1}^3(\D{A_{tv}}{\nu_u} -
               \D{A_{tu}}{\nu_v})\Bigr),
    \end{split}
  \end{equation*}
  which reduces to an expression that only depends on the symmetric
  part of~\( A \), so we may take \( A \) to be symmetric.
\end{proof}

Note that the right-hand side of~\eqref{eq:Vpot-eq} is not elliptic,
so a rewriting of Theorem~\ref{thm:toric-G2-charac} looses ellipticity
of that system.
The papers~\cite{Eastwood:elasticity2, Eastwood:elasticity1} contain a
description of the kernel of \( A \mapsto V(A) \).

\subsection{Digression: natural PDEs for toric
\texorpdfstring{\( \G_2 \)}{G2}-manifolds}
\label{sec:natural-eq}

As we have already seen, toric \( \G_2 \)-manifolds come with an
associated action of \( \GL(3,\bR) \).
Thus a way of approaching equation~\eqref{eq:elliptic}, is to
understand how \( L \) and \( Q \) transform with respect to this
action.

The general linear group~\( \GL(3,\bR) \) acts by changing the basis
of \( \lie t^3 \) and so of \( \xi_p\cong\bR^3 \), \( p \in M_0 \).
It is useful to write \( \GL(3,\bR)\cong\bR^\times\times\SL(3,\bR) \)
and accordingly express irreducible representations in the form
\( \ell^p\Gamma_{a,b} \), where \( \Gamma_{a,b} \) is an irreducible
representation of \( \SL(3,\bR) \) (see, e.g.,
\cite{Baston-E:Penrose}) and \( \ell \) is the standard
one-dimensional representation of \( \bR^\times\to\bR\setminus\{0\} \)
given by \( t\mapsto t \).
As an example, this means that we have for \( p\in M_0 \) that
\( \xi_p=\ell^1\Gamma_{0,1} \).

So let \( U = (\bR^3)^*=\ell^{-1}\Gamma_{1,0} \), viewed as a
representation of \( \GL(3,\bR) \).
Then \( V \in S^2(U)=\ell^{-2}\Gamma_{2,0} \).
The collection of first order partial derivatives
\( V^{(1)} = (V_{ij,k}) = (\partial V_{ij}/\partial\nu_k) \) is then
an element of
\( S^2(U)\otimes \ell^{-3} U^* =\ell^{-4}\Gamma_{2,0} \otimes
\Gamma_{0,1} \).
As a \( \GL(3,\bR) \) representation this decomposes as
\begin{equation*}
  S^2(U)\otimes \ell^{-3} U^* = \ell^{-4} \Gamma_{1,0} \oplus \ell^{-4} \Gamma_{2,1},
\end{equation*}
with the projection to \( \Gamma_{1,0} \) being just the contraction
\( S^2(\Gamma_{1,0})\otimes \Gamma_{0,1}\to \Gamma_{1,0} \), and
\( \Gamma_{2,1} \) denoting the kernel of this map.
The divergence-free equation~\eqref{eq:div-free} just says this
contraction is zero, so \( V^{(1)} \in \ell^{-4} \Gamma_{2,1} \).

The operator \( Q \) is a symmetric quadratic operator on
\( V^{(1)} \) with values in \( S^2(U) \).
Thus we may think of \( Q(dV) \) as an element of the space
\( \ell^6S^2(\Gamma_{2,1})^* \otimes S^2(\Gamma_{1,0}) \).
This space contains exactly one submodule isomorphic to
\( \ell^6 \) as \( S^2(\Gamma_{1,0})^* \) is a submodule of
\( S^2(\Gamma_{2,1})^* \).
Direct computations show that \( Q(dV) \) belongs to \( \ell^6 \).

Similarly, we may discuss the second order terms
in~\eqref{eq:elliptic}.
We have
\( V^{(2)} = (V_{ij,k\ell}) \in R = (S^2(U) \otimes S^2(\ell^{-3}U^*))
\cap (\ell^{-6}\Gamma_{2,1}\otimes \Gamma_{0,1}) \).
Now, ignoring the \( \partial^2V/\partial\mu^2 \) term, \( L(V) \) is
built from a product of \( V \) with \( V^{(2)} \) and takes values in
\( S^2(U) \).
So \( L(V) \in S^2(U)^*\otimes R^* \otimes S^2(U) \).
In this case, there are two submodules isomorphic to \( \ell^6 \),
but only one appears in \( L(V) \), corresponding to the contractions
\begin{equation*}
  \acontraction{}{S^2(U^*)}{\otimes\Bigl( \overline{S^2(U^*)}\otimes }{S^2(\ell^3U)}
  \acontraction[2ex]{S^2(U^*)\otimes\Bigl(}{S^2(U^*)}{\otimes S^2(\ell^3U)\Bigr)\otimes}{S^2(U)}
  S^2(U^*)\otimes\Bigl( S^2(U^*)\otimes S^2(\ell^3U)\Bigr)\otimes S^2(U)\to \ell^6.
\end{equation*}
Contracting in this way is arguably the most natural choice.

Finally, addressing the terms of \( L \) involving
\( \partial^2V/\partial\mu^2 \), we have that
\( \partial/\partial \mu \) is an element of \( \ell^{-3} \), and
therefore \( \partial^2V/\partial\mu^2 \) belongs to
\( \ell^6S^2(U)^*\otimes S^2(U) \).
In fact, it is easy to see that \( \partial^2V/\partial\mu^2 \)
belongs to the one-dimensional summand isomorphic to
\( \ell^6 \) as we are tracing.

In conclusion, we have that \( L \) and \( Q \) are preserved up to
scale by \( \GL(3,\bR) \) change of basis, and this specifies \( Q \)
uniquely.

\begin{proposition}
  Under the action of \( \GL(3,\bR) \), \( L(V) \) and \( Q(dV) \)
  transform as elements of \( \ell^6 \).
  Moreover, up to scaling, \( Q \) is the unique \( S^2(U) \)-valued
  quadratic form in \( dV \) with this property.  \qed
\end{proposition}

\section{Behaviour near singular orbits}
\label{sec:sing-orb}

In our description of toric \( \G_2 \)-manifolds, we have so far been
focusing on the regular part \( M_0 \subset M \).
We now turn to address what happens near a singular orbit for the
\( T^3 \)-action.

\subsection{Flat models}
\label{sec:flat-models}

For a complete hyperK\"ahler manifold with a tri\hyphen Hamiltonian
action of \( T^n \) it is known that the hyperK\"ahler moment map
induces a homeomorphism \( M/T^n \to \bR^n \)
(see~\cite{Dancer-S:hK,Swann:twist-vs-mod}).
In this section, we establish the analogous result for toric
\( \G_2 \)-manifolds for flat models with a singular orbit; later we
will prove this in general.
There are two cases to consider as the singular orbit can be either
\( S^1 \) or \( T^2 \), corresponding to a stabiliser of dimension
\( 2 \) or~\( 1 \).

\subsubsection{Two-dimensional stabiliser}
\label{sec:two-dimens-stab}

Consider the flat model \( M = S^1\times\bC^3 \) equipped with the
\( 3 \)-form
\begin{equation*}
  \varphi = \tfrac i2 dx\wedge(dz_1\wedge d\overline z_1+dz_2\wedge
  d\overline z_2+dz_3\wedge d\overline z_3)+\re (dz_1\wedge dz_2\wedge
  dz_3),
\end{equation*}
with dual \( 4 \)-form
\begin{equation*}
  \Hodge\varphi = \im(dz_1\wedge dz_2\wedge dz_3) \wedge dx-
  \tfrac18(dz_1\wedge d\overline z_1+dz_2\wedge d\overline
  z_2+dz_3\wedge d\overline z_3)^2,
\end{equation*}
where \( z_j = x_j+iy_j \), \( j=1,2,3 \), are standard complex
coordinates on~\( \bC^3 \).

There is a natural effective \( T^3 \)-action on \( M \): writing
\( T^3=S^1\times T^2 \), the \( T^2 \) acts as a maximal torus of
\( \SU(3) \) on \( \bC^3 \) and the remaining circle acts naturally on
the \( S^1 \) factor.
Correspondingly, we have generating vector fields given by
\begin{equation*}
  U_1 = \D{}x,\
  U_2= 2\re\Bigl(
  i\Bigl(z_1\D{}{z_1}-z_3\D{}{z_3}\Bigr)
  \Bigr),\
  U_3= 2\re\Bigl(i\Bigl(z_2\D{}{z_2}-z_3\D{}{z_3}\Bigr)\Bigr).
\end{equation*}
It follows that the matrix \( B \) is
\begin{equation*}
  B =
  \begin{pmatrix}
    1 & 0 & 0 \\
    0 & \abs{z_1}^2 + \abs{z_3}^2 & \abs{z_3}^2\\
    0 & \abs{z_3}^2 & \abs{z_2}^2 + \abs{z_3}^2
  \end{pmatrix}
\end{equation*}
and so \( V \) takes the form
\begin{equation*}
  V=
  \begin{pmatrix}
    1 & 0 & 0\\
    0 & (\abs{z_2}^2+ \abs{z_3}^2)/A& - \abs{z_3}^2/A\\
    0 & -\abs{z_3}^2/A& (\abs{z_1}^2+\abs{z_3}^2)/A
  \end{pmatrix}
  ,
\end{equation*}
where \( A = \abs{z_1z_2}^2 + \abs{z_3z_1}^2 + \abs{z_2z_3}^2 \).
We have that \( M_0 \) is the complement of the following sets:
\( M^{T^2}=S^1 \times \Set{0} \) where the singular stabiliser is
\( T^2 = \Set{1} \times T^2 \leqslant S^1 \times T^2 = T^3 \);
\( M^{S^1_i} = S^1 \times \Set{z_j=z_k=0,\;z_i \ne 0} \),
\( (ijk) = (123) \), which all have singular stabiliser circles
\( S^1_i\leqslant T^2 \leqslant T^3 \).

For the multi-moment maps, we first compute
\begin{equation*}
  d\mu = U_1\wedge U_2\wedge U_3 \hook \Hodge\varphi=
  d\im(z_1z_2z_3),
\end{equation*}
giving that, up to addition of a constant, \( \mu = \im(z_1z_2z_3) \).
Similarly, we find \( \nu_1 = -\re(z_1z_2z_3) \), from
\( U_2\wedge U_3\hook\varphi \), and
\begin{equation*}
  d\nu_2 = U_3\wedge U_1 \hook \varphi
  =\tfrac12 d(\abs{z_2}^2-\abs{z_3}^2).
\end{equation*}
So, again up to addition of a constant,
\( \nu_2 = \tfrac12 (\abs{z_2}^2-\abs{z_3}^2) \).
Finally, we have that
\( \nu_3= -\tfrac12(\abs{z_1}^2 - \abs{z_3}^2) \).
Summarising, the multi-moment maps are
\begin{equation*}
  \nu_1 + i\mu = - \overline{z_1z_2z_3},\quad
  \nu_2 = \tfrac12(\abs{z_2}^2-\abs{z_3}^2),\quad
  \nu_3 = - \tfrac12(\abs{z_1}^2-\abs{z_3}^2).
\end{equation*}

\begin{proposition}
  \label{prop:hom-flat-2d-stab}
  The multi-moment map
  \( (\nu,\mu)\colon S^1 \times \bC^3 \to \bR^{3}\times \bR = \bR^4 \)
  induces a homeomorphism
  \( (S^1\times \bC^3)/T^3 = \bC^3/T^2 \to \bR^4 \).
\end{proposition}

As the referee points out, this map \( \bC^{3}/T^{2}\to\bR^{4} \) has
also been considered in~\cite{Aganagic-KMV:top-vertex}.

\begin{proof}
  Let us introduce some new variables.
  Putting \( t = \abs{z_3}^2 \), we have \( \abs{z_1}^2 = t - a \),
  \( \abs{z_2}^2 = t - b \), where \( a = 2\nu_3 \) and
  \( b = -2\nu_2 \).
  For
  \( c = \abs{\mu}^2+\abs{\nu_1}^2 = \abs{z_1}^2\abs{z_2}^2\abs{z_3}^2
  \), we have the relation
  \begin{equation*}
    f(t) \coloneqq t(t-a)(t-b) = c.
  \end{equation*}
  Note that \( f \) has zeros at \( 0 \), \( a \) and~\( b \).
  The constraints \( \abs{z_i}^2 \geqslant 0 \), imply
  \( t \geqslant x \coloneqq \max\Set{0,a,b} \).
  Now \( f(t) \to \infty \) as \( t \to \infty \), so
  \( f(\halfopen{x,\infty}) = \halfopen{0,\infty} \) and \( f \) is
  strictly monotone increasing on \( \halfopen{x,\infty} \).
  Thus \( f(t) = c \) has a unique solution
  \( t = t(a,b,c) \geqslant x \) for each \( a,b \in \bR \) and each
  \( c\geqslant 0 \).

  Write \( \rho\colon \bC^3/T^2 \to \bR^4 \) for the map induced by
  \( (\nu,\mu) \).
  Given \( (p,q) \in \bR^3 \times \bR = \bR^4 \), let
  \( t = t(2p_3,-2p_2,q^2+p_1^2) \), where \( t(a,b,c) \) is defined
  above.
  Now \( \rho(z_1,z_2,z_3) = (p,q) \) if and only if
  \( (\abs{z_1}^2,\abs{z_2}^2,\abs{z_3}^2) = (t-2p_3,t+2p_2,t) \) and
  \( z_1z_2z_3 = (iq-p_1) \).
  One sees that these equations are consistent, \( \rho \)~is
  surjective, and solutions are unique up to the action of
  \( T^2 \leqslant \SU(3) \).
  Thus \( \rho \) is a continuous bijection \( \bC^3/T^2 \to \bR^4 \).

  But \( \bC^3/T^2 \) is homeomorphic to \( \bR^4 \).
  Indeed, it follows from the results of~\cite{Hughes-S:quotients}
  that \( S^5/T^2 \) is homeomorphic to \( S^3 \), so the claimed
  result follows by considering the cones on these spaces.

  To be explicit, we note that
  \( S^5 = \Set{(z_1,z_2,z_3) \with \abs{z_1}^2 + \abs{z_2}^2 +
  \abs{z_3}^2 = 1} =
  \Set{(t_1^{1/2}e^{iu},t_2^{1/2}e^{iv},t_3^{1/2}e^{iw}) \with t_i
  \geqslant 0,\; t_1+t_2+t_3=1} \) with \( T^2 \)-action induced by
  \( (e^{i\theta},e^{i\phi})\cdot(e^{iu},e^{iv},e^{iw}) =
  (e^{i(\theta+u)},e^{i(\phi+v)},e^{i(w-\theta-\phi)}) \).
  Each \( T^2 \)-orbit contains a representative with \( u=v=w \).
  Furthermore, this representative is unique modulo \( 2\pi/3 \)
  unless some \( t_i \) is zero, since
  \( \theta+u=\phi+v=w-\theta-\phi \pmod{2\pi} \) implies the common
  value \( a \) satisfies \( 3a = u+v+w \pmod{2\pi} \) and each such
  \( a \) gives a unique solution for \( \theta \) and \( \phi \)
  mod~\( 2\pi \).

  Topologically the two-simplex
  \( \Set{(t_1,t_2,t_3) \with t_i \geqslant 0,\; t_1+t_2+t_3=1} \) is
  a unit disc \( \Set{ w \in \bC \with \abs{w}^2 \leqslant 1} \).
  The quotient \( S^5/T^2 \) has circle fibres over the interior of
  the disc that collapse to points on the boundary.
  Thus \( S^5/T^2 \) is topologically
  \( \Set{(z,w) \in \bC^2 \with \abs{z}^2+\abs{w}^2=1} = S^3 \).

  Now \( \rho \) is a continuous bijection
  \( \bR^4 = \bC^3/T^2 \to \bR^4 \).
  By Brouwer's invariance of domain
  (see~\cite[Thm.~7.12]{Madsen-T:deRham}), it follows that
  \( \rho \)~is a homeomorphism.
\end{proof}

\subsubsection{One-dimensional stabiliser}
\label{sec:one-dimens-stab}

The previous model contains points with stabiliser~\( S^1 \), but we
can also provide a simple standard model in this case.
Let \( M = (T^2 \times \bR) \times \bC^2 \) with the \( 3 \)-torus
split as \( T^3 = T^2 \times S^1 \), the first \( T^2 \)-factor acting
on the corresponding torus in the first factor of \( M \), and the
\( S^1 \)-factor acting as the maximal torus of \( \SU(2) \)
on~\( \bC^2 \).
Introduce standard (local) coordinates \( x,y,u \) for
\( T^2\times \bR \) and \( (z,w) \) for~\( \bC^2 \).

The \( \G_2 \) \( 3 \)-form may be written as
\begin{equation*}
  \begin{split}
    \varphi
    &= du \wedge dx \wedge dy - du \wedge \tfrac i2(dz \wedge
      d\overline z + dw \wedge d\overline w) \eqbreak
      - \re((dx-idy)\wedge dz \wedge dw),
  \end{split}
\end{equation*}
with dual \( 4 \)-form
\begin{equation*}
  \begin{split}
    \Hodge\varphi
    &= \tfrac18 (dz\wedge d\overline z + dw \wedge d\overline w)^2
      + dx \wedge dy \wedge \tfrac i2(dz \wedge
      d\overline z + dw \wedge d\overline w) \eqbreak
      + du \wedge \im ((dx - i dy) \wedge dz \wedge dw).
  \end{split}
\end{equation*}
The generating vector fields are then
\begin{equation*}
  U_1 = \D{}x,\quad U_2 = \D{}y,\quad
  U_3 = - 2\re\Bigl(i\Bigl(z\D{}z - w\D{}w\Bigr)\Bigr).
\end{equation*}
The matrix \( V \) is now
\begin{equation*}
  \begin{pmatrix}
    1 & 0 & 0 \\
    0 & 1 & 0 \\
    0 & 0 & 1/(\abs z^2+\abs w^2)
  \end{pmatrix}.
\end{equation*}

We compute the multi-moment maps:
\begin{align*}
  d\mu &= U_1 \wedge U_2 \wedge U_3 \hook \Hodge\varphi
         = d(\tfrac12(\abs z^2-\abs w^2)), \\
  d\nu_1 &= U_2 \wedge U_3 \hook \varphi
           = d \re(zw), \\
  d\nu_2 &= U_3 \wedge U_1 \hook \varphi
           = d\im(zw), \\
  d\nu_3 &= U_1 \wedge U_2 \hook \varphi
           = du.
\end{align*}
Thus, we may take
\begin{equation*}
  \mu = \tfrac12(\abs z^2-\abs w^2),\quad
  \nu_1 + i\nu_2 = zw,\quad
  \nu_3 = u.
\end{equation*}
Note that, as expected, \( (\mu,\nu_1,\nu_2) \) are just the standard
hyperK\"ahler moment maps for the action of \( S^1 \) on
\( \bH = \bC^2 \).
We know that this is essentially the Hopf fibration \( S^3 \to S^2 \)
on distance spheres in \( \bH = \bR^4 \) and \( \bR^3 \).  Indeed
\begin{equation*}
  \mu^2 + \nu_1^2 + \nu_2^2 = \tfrac14 (\abs z^4 - 2\abs z^2\abs w^2 +
  \abs w^4) + \abs z^2\abs w^2 = \tfrac14 (\abs z^2 + \abs w^2)^2
\end{equation*}
so \( 3 \)-spheres of radius \( r \) are mapped to \( 2 \)-spheres of
radius \( r^2/2 \).  Again we get:

\begin{proposition}
  \label{prop:hom-flat-1d-stab}
  The multi-moment map
  \( (\nu,\mu)\colon (T^2 \times\bR)\times \bC^2 \to \bR^4 \) induces
  a homeomorphism
  \( ((T^2\times \bR)\times \bC^2)/T^3 = \bR \times \bH/S^1 \to \bR^4
  \).  \qed
\end{proposition}

\subsection{Comparing with the flat models}
\label{sec:flat-mod-comp}

We now turn to general toric \( \G_2 \)-manifolds \( (M,\varphi) \).
One way of obtaining a first feel for the behaviour of the
multi-moment maps near singular stabilisers is by comparing with the
flat models.
In order to do so, it turns out useful to recall some basic facts
about Killing fields.

\subsubsection{Killing vector fields}
\label{sec:kill-vect-fields}

If a vector field \( X \) on \( (M,g) \) is Killing, then this implies
that \( \nabla X \) is skew-adjoint, normalises the holonomy algebra
and
\begin{equation*}
  \nabla^2_{A,B} X = -R_{X,A}B.
\end{equation*}
For the last result, cf.~\cite{Kobayashi:transformation} (see
also~\cite{Besse:Einstein}), we use that \( X \) preserves the
Levi-Civita connection,
\begin{equation}
  \label{eq:Killing-LC}
  [X,\nabla_A B]
  = \nabla_{[X,A]}B + \nabla_A[X,B]
  = \nabla_{[X,A]}B + \nabla_A\nabla_X B - \nabla_A\nabla_B X
\end{equation}
to get
\begin{equation*}
  \begin{split}
    R_{X,A}B
    &= \nabla_X\nabla_A B - \nabla_A\nabla_X B - \nabla_{[X,A]}B
      = \nabla_X\nabla_A B - [X,\nabla_A B] - \nabla_A\nabla_B X \\
    &= \nabla_{\nabla_A B}X - \nabla_A\nabla_B X
      = - \nabla^2_{A,B}X.
  \end{split}
\end{equation*}

It follows that at a zero \( p \) of \( X \), we have
\( (\nabla^2X)_p = 0 \) and
\begin{equation*}
  \begin{split}
    (\nabla^3_{A,B,C}X)_p
    &= (-(\nabla_A(R_X))_B C)_p
      = (- (\nabla_A R)_{X,B}C - R_{\nabla_A X,B}C)_p \\
    &= -(R_{\nabla_A X,B}C)_p.
  \end{split}
\end{equation*}
Note also that at such a \( p \), the endomorphism \( (\nabla X)_p \)
on \( T_p M \) gives the infinitesimal action of the one-parameter
group generated by~\( X \).

If \( X \) and \( Y \) are two commuting Killing vector fields with
\( X_p = 0 \), then we claim that the endomorphisms \( \nabla X \) and
\( \nabla Y \) commute at~\( p \).
To see this, let \( A \) be an arbitrary vector field.
Then at \( p \), we have \( \nabla_X\any = 0 \), so
using~\eqref{eq:Killing-LC} gives
\begin{equation*}
  \begin{split}
    [\nabla X,\nabla Y]_{p}(A)
    &= (\nabla_{\nabla_A Y}X - \nabla_{\nabla_A X}Y)_{p}
      = ([\nabla_A Y,X] - \nabla_{\nabla_A X}Y)_{p} \\
    &= (\nabla_{[A,X]}Y + \nabla_A[Y,X] - \nabla_{\nabla_A X}Y)_{p}
      = (\nabla_{\nabla_X A}Y)_{p} = 0,
  \end{split}
\end{equation*}
as claimed.

Finally, for a vector field \( X \) preserving \( \varphi \), we get
that \( X \) is Killing and
\begin{equation*}
  \begin{split}
    0
    &= L_X\varphi
      = d(X\hook \varphi)
      = \mathbf a \varphi(\nabla X,\any,\any) \\
    &= \varphi(\nabla X,\any,\any)
      + \varphi(\any,\nabla X,\any)
      + \varphi(\any,\any,\nabla X),
  \end{split}
\end{equation*}
which shows that \( \nabla X \in \lie g_2 \).

\subsubsection{Near points with two-dimensional stabiliser}
\label{sec:near-points-with}

Let \( p \in M \) be a point with \( \Stab_{T^3}(p) \cong T^2 \).
We may identify \( T_p M \) linearly with
\( \bR \times \bC^3 = T_{(1,0)}(S^1\times \bC^3) \) in the standard
model of~\S\ref{sec:two-dimens-stab}, so that the \( \G_2 \)-forms
agree at this point.
We have an equivariant diffeomorphism between a neighbourhood of
\( 0 \in T_p M \) and a neighbourhood of \( p\in M \) via the local
tubular model
\( T^3 \times_{\Stab(p)} \bC^3 \cong T^3/T^2 \times \bC^3 \), the map
on the \( \bC^3 \) part being given by the Riemannian exponential map.
The elements of \( \Stab(p) \) act on \( \bR \times \bC^3 \) linearly
as a maximal torus in \( \SU(3) \).
We may choose our linear identification so this is the standard
diagonal subgroup and may choose our generators \( U_2 \), \( U_3 \)
for \( \Stab(p) \) so that
\begin{equation*}
  (\nabla U_2)_p = \diag(i,0,-i),\quad (\nabla U_3)_p = \diag(0,i,-i)
\end{equation*}
in this model.

Let us now specify a choice of \( U_1 \).
We note that the \( T^3 \)-orbit of \( p \) is
\( T^3/\Stab(p) \times \Set{0} \) in the local model.
This orbit is the fixed point set of \( \Stab(p) \), so is totally
geodesic.
For any \( U \) generating \( T^3/\Stab(p) \), we thus have
\( (\nabla_U U)_p \in \bR U \).
But \( (\nabla U)_p \) is an element of~\( \lie g_2 \subset \so(7) \),
so \( (\nabla_U U)_p = 0 \).
As the splitting \( \bR \times \bC^3 \) is orthogonal, it follows that
\( (\nabla U)_p \in \su(3) \).
Now each \( U_i \) vanishes at~\( p \), so the endomorphisms
\( (\nabla U_i)_p \) commute with \( (\nabla U)_p \),
by~\S\ref{sec:kill-vect-fields}.
As \( (\nabla U_2)_p \), \( (\nabla U_3)_p \) generate a maximal torus
of \( \su(3) \), it follows that
\( (\nabla U)_p = a(\nabla U_2)_p + b(\nabla U_3)_p \), for some
\( a,b \in \bR \).
Putting \( U_1 = U - aU_2 - bU_3 \), we still have that \( U_1 \)
generates \( T^3/\Stab(p) \) and get \( (\nabla U_1)_p = 0 \).
If we wish, we may assume that \( (U_1)_p \) is of length~\( 1 \).

Now consider the multi-moment maps.  For \( \nu_2 \), we have
\begin{equation*}
  (\nabla\nu_2)_p
  = (d\nu_2)_p
  = (U_3 \wedge U_1 \hook \varphi)_p = 0,
\end{equation*}
since \( (U_3)_p = 0 \).
Similarly \( \nabla\nu_3 = 0 = \nabla\nu_1 = \nabla\mu \) at \( p \).
Furthermore,
\begin{equation*}
  \begin{split}
    (\nabla^2\nu_2)_p
    &= \bigl( (\nabla\varphi)(U_3,U_1,\any) + \varphi(\nabla U_3, U_1,
      \any) + \varphi(U_3, \nabla U_1, \any) \bigr)_p \\
    &= \varphi(\nabla U_3, U_1, \any)_p
  \end{split}
\end{equation*}
agrees with the flat model at~\( p \).
Similarly for \( (\nabla^2\nu_3)_p \).  For \( \nu_1 \), we have
\begin{equation*}
  (\nabla^2\nu_1)_p
  = \bigl( \varphi(\nabla U_2, U_3, \any) + \varphi(U_2, \nabla U_3,
  \any) \bigr)_p
  = 0,
\end{equation*}
as both \( U_2 \) and \( U_3 \) vanish at~\( p \).
Similarly, \( (\nabla^2\mu)_p=0 \).

For third order derivatives, we have
\begin{equation*}
  (\nabla^3\nu_2)_p
  = \bigl(\varphi(\nabla^2 U_3, U_1,\any) + 2\varphi(\nabla U_3,
  \nabla U_1, \any) + \varphi(U_3, \nabla^2 U_1, \any) \bigr)_p
  = 0,
\end{equation*}
since \( (\nabla^2U_3)_p = 0 \) by~\S\ref{sec:kill-vect-fields}, and
\( (\nabla U_1)_p=0 \) by our choice of \( U_1 \).
Similarly, \( (\nabla^3\nu_3)_p=0 \).  On the other hand,
\begin{equation*}
  \begin{split}
    (\nabla^3\nu_1)_p
    &= \bigl(\varphi(\nabla^2 U_2, U_3,\any) + 2\varphi(\nabla U_2,
      \nabla U_3, \any) + \varphi(U_2, \nabla^2 U_3, \any) \bigr)_p \\
    &= 2\varphi(\nabla U_2, \nabla U_3, \any)_p,
  \end{split}
\end{equation*}
which agrees with the flat model, as does \( (\nabla^3\mu)_p \).

Let us now compute fourth order derivatives.  Firstly,
\begin{equation*}
  \begin{split}
    (\nabla^4\nu_2)_p
    &= \bigl(\varphi(\nabla^3 U_3, U_1,\any)
      + 3\varphi(\nabla^2 U_3, \nabla U_1, \any)
      \eqbreak
      + 3\varphi(\nabla U_3, \nabla^2 U_1, \any)
      + \varphi(U_3, \nabla^3 U_1, \any) \bigr)_p \\
    &= \varphi(\nabla^3 U_3, U_1,\any)_p
      + 3\varphi(\nabla U_3, \nabla^2 U_1, \any)_p \\
    &= -\varphi(R_{\nabla U_3,\any}\any, U_1, \any)_p - 3\varphi(\nabla
      U_3, R_{U_1,\any}\any, \any)_p,
  \end{split}
\end{equation*}
with a similar expression for \( (\nabla^4\nu_3)_p \).
For \( \nu_1 \) and \( \mu \), the same type of computation gives
\( (\nabla^4\nu_1)_p = 0 = (\nabla^4\mu)_p \).
In conclusion, we have shown:

\begin{lemma}
  Let \( p\in M \) be a point with stabiliser \( T^2\) whose
  infinitesimal generators are \( U_2,U_3 \).
  Then the multi-moment maps \( \nu_2,\nu_3 \) agree with the flat
  model to order \( 3 \) and \( \nu_1,\mu \) agree with the flat model
  to order \( 4 \).  \qed
\end{lemma}

\subsubsection{Near points with one-dimensional stabiliser}
\label{sec:near-points-1d}

In this case, we need less detailed information.
Let \( p \in M \) have \( \Stab_{T^3}(p) \cong S^1 \).
We take the infinitesimal generator for this stabiliser to
be~\( U_3 \).
Let \( U_1 \) and \( U_2 \) be two vector fields of the \( T^3 \)
action that generate the quotient \( T^3/\Stab(p) \cong T^2 \).
We take them to be of unit length and orthogonal at~\( p \).
Then \( U_1 \) and \( U_2 \) are invariant under \( U_3 \) as is their
\( \G_2 \)-cross-product
\( U_1 \times U_2 = \varphi(U_1,U_2,\any)^\sharp \).
We have \( T_p M = \bR^3 \times \bC^2 \) linearly, with
\( \bR^3 = \Span{U_1,U_2,U_1\times U_2}_p \) and \( \bC^2 \) the
orthogonal complement.
This identification may be chosen so that \( (\nabla U_3)_p \) acts as
the element \( \diag(i,-i) \) in \( \su(2) \) on~\( \bC^2 \).
The local model is
\( T^3 \times_{\Stab(p)} (\bR \times \bC^2) \cong (T^2 \times \bR)
\times \bC^2 \), with \( T^2 \times \bR \times \Set{0} \) the fixed
point set of~\( U_3 \), so totally geodesic.
Now \( d\nu_3 = (U_1\times U_2)^\flat \) is non-zero and therefore
provides a transverse coordinate to a six-dimensional level set, and
\( d\nu_1 = 0 = d\nu_2 = d\mu \) are zero at~\( p \).
The three second derivatives \( \nabla^2\nu_1 \), \( \nabla^2\nu_2 \)
and \( \nabla^2\mu \) are specified by \( U_i \), \( i=1,2 \), and
\( \nabla U_3 \) at~\( p \) and so all agree with the standard flat
model at~\( p \).

\subsubsection{Images of singular orbits}

First consider a point \( p \) with stabiliser~\( S^1 \).
The previous section provides an integral basis \( U_1,U_2,U_3 \) of
\( \lie t^3 \) with \( (U_3)_p = 0 \).
Furthermore, this is true for all points of \( T^2 \times \bR \) in
the local model.
It follows that \( \nu_1 \), \( \nu_2 \) and \( \mu \) are constant on
this set, and so the image under \( (\nu,\mu) \) of this family of
singular orbits is a straight line parameterised by the values
of~\( \nu_3 \).

Now for points \( p \) with \( T^2 \)-stabiliser, these lie on a
circle \( T^3p \).
The normal bundle is modelled on \( \bC^3 \) and there are three
families of points with stabiliser \( S^1 \).
These families meet at \( p \) and correspond to the complex
coordinate axes in \( \bC^3 \).
There is thus an integral basis \( U_1,U_2,U_3 \) of \( \lie t^3 \)
with \( U_2 = 0 = U_3 \) at \( p \) and such that \( U_2 \), \( U_3 \)
and \( - U_2 - U_3 \) generate the \( S^1 \) stabilisers of the three
families.
The images of the families under \( (\nu,\mu) \) all have the same
constant \( \mu \)- and \( \nu_1 \)-coordinates, and provide the three
half-lines meeting at the image of \( p \) lying in \( \nu_3 \),
\( \nu_2 \) or \( (\nu_2-\nu_3) \) constant.

Summarising, we have:

\begin{lemma}
  \label{lem:sing-point}
  For \( p\in M\setminus M_0 \), we have \( \rank B_p \leqslant 2 \).
  The image in \( M/T^3 \) of the union \( M\setminus M_0 \) of
  singular orbits consists of trivalent graphs lying in sets
  \( \mu = \text{constant} \) with edges that are straight lines of
  rational slope in the \( \nu \)-coordinates.  At each vertex the three
  primitive integral slope vectors sum to zero, in particular these
  edges lie a plane.  \qed
\end{lemma}

\subsection{Deforming to the flat model}
\label{sec:deforming-flat-model}

Let \( \varphi \) be a torsion-free \( \G_2 \)-structure on the ball
\( B_2(0) \subset \bR^7 \) with centre~\( 0 \) and radius~\( 2 \).
Choose linear coordinates \( (x_1,\dots,x_7) \) on \( \bR^7 \) so that
\( \varphi|_0 = \varphi_0|_0 \), where \( \varphi_0 \) is the standard
constant coefficient \( \G_2 \)-form.
Our aim is to construct a family of torsion-free \( \G_2 \)-structures
\( \varphi_t \), \( t\in\halfclosed{0,1} \), with
\( \varphi_1 = \varphi \), and with \( \varphi_t \) converging to
\( \varphi_0 \) on \( \overline{B_1(0)} \) in each \( C^k \)-norm.

For \( t \in \halfclosed{0,1} \), define a linear diffeomorphism
\( \lambda_t\colon \bR^7 \to \bR^7 \) by \( \lambda_t(x) = tx \).
Note that \( \lambda_t^*\varphi_0 = t^3\varphi_0 \), so let us take
\( \varphi_t \) to be
\begin{equation*}
  \varphi_t = t^{-3}\lambda_t^*\varphi,
  \eqcond{for \( t\in\halfclosed{0,1} \).}
\end{equation*}

We have \( \varphi = \varphi_0 + \psi \) where
\( \psi \in \Omega^3(B_2(0)) \) is smooth and has \( \psi|_0 = 0 \).
It follows that
\begin{equation*}
  \psi = \sum_{\abs{I}=3} f_I dx_I,
\end{equation*}
where \( dx_I = dx_{i_1}\wedge dx_{i_2} \wedge dx_{i_3} \), for
\( I = (i_1,i_2,i_3) \in \Set{1,\dots,7}^3 \), and \( f_I \) is smooth
with \( f_I(0) = 0 \).
We may therefore write \( f_I(x) = \sum_{k=1}^7 x_k h_{I,k}(x) \) with
\( h_{I,k} \) smooth.
We have \( \lambda_t^*\psi = \sum_I (\lambda_t^*f_I) t^3dx_I \) and
\( (\lambda_t^*f_I)(x) = \sum_k tx_k h_{I,k}(tx) \), so
\( \norm{\lambda_t^*f_I}_{C^0} \leqslant t \norm{f_I}_{C^0} \).
Thus putting \( \psi_t = t^{-3}\lambda_t^*\psi \), so
\( \varphi_t = \varphi_0 + \psi_t \), we get
\( \norm{\psi_t}_{C^0} \leqslant t\norm{\psi}_{C^0} \).
Thus \( \varphi_t \to \varphi_0 \) in~\( C^0(\overline{B_1(0)}) \) as
\( t \searrow 0 \).

The Riemannian metric \( g_t \) defined by \( \varphi_t \) satisfies
\begin{equation*}
  g_t = t^{-2}\lambda_t^*g,
\end{equation*}
where \( g = g_1 \).
The same types of computations as above show that
\( g_t \to g_0 = \sum_{i=1}^7 dx_i^2 \) in \( C^0 \) as
\( t \searrow 0 \).
Let \( \nabla^t \) be the Levi-Civita connection of~\( g_t \) and
write its Christoffel symbols as~\( (\Gamma_t)_{ij}^k \).
We claim that \( \nabla^t \to \nabla^0 \), meaning that
\( (\Gamma_t)_{ij}^k \to 0 \), as \( t \searrow 0 \).

We have
\begin{equation*}
  (g_t)_{ij}(x) = \delta_{ij} + t \sum_{k=1}^7 x_k h_{ijk}(tx)
\end{equation*}
for some smooth functions~\( h_{ijk} \).  Thus
\begin{equation*}
  \D{}{x_\ell}(g_t)_{ij}(x) = t h_{ij\ell}(tx) + t^2 \sum_{k=1}^7 x_k
  \D{h_{ijk}}{x_\ell}(tx)
\end{equation*}
and
\begin{equation*}
  (g_t^{-1})_{ij}(x) = \delta_{ij} + t\sum_{k=1}^7 x_k\tilde h_{ijk}(tx),
\end{equation*}
for some smooth functions \( \tilde h_{ijk} \). This gives
\begin{equation*}
  2(\Gamma_t)_{ij}^k(x)
  = t(h_{ij\ell} + h_{ji\ell} - h_{\ell ij})(tx) + O(t^2)
\end{equation*}
and hence \( (\Gamma_t)_{ij}^k \to 0 \), as claimed.

Now note that
\( 0 = \nabla^t\varphi_t = \nabla^t\varphi_0 + \nabla^t\psi_t, \) so
\( \nabla^t\psi_t = - \nabla^t\varphi_0 \to 0 \) in \( C^0 \) as
\( t \searrow 0 \).
It follows that \( \varphi_t \to \varphi_0 \) in \( C^1 \).
Iterating, noting that each derivative adds an extra factor
of~\( t \), we get the claimed convergence in \( C^k \).

If \( U \) is a linear symmetry of \( \bR^7 \) that preserves
\( \varphi \), then it is also a symmetry of \( \varphi_t \), since
\( U \) commutes with dilations.
Furthermore, if \( X = \sum_{i=1}^7 v_i\partial/\partial x_i \) is a
constant coefficient vector field preserving \( \varphi \) then it is
also a symmetry of \( \varphi_t \).
Indeed, the one-parameter group generated by \( X \) is
\( T_s(x) = x + sv \).
Now, for any \( f\in C^\infty(V) \) we have \( (\lambda_t)_*X = tX \).
This gives
\begin{equation*}
  \begin{split}
    L_X\varphi_t
    &= t^{-3} L_X \lambda_t^*\varphi
      = t^{-3} \bigl( X \hook d\lambda_t^*\varphi +
      d(X\hook\lambda_t^*\varphi)\bigr) \\
    &= t^{-2} \lambda_t^* L_X\varphi
      = 0,
  \end{split}
\end{equation*}
which is the claimed symmetry.

Note that we now also get that the multi-moment maps converge to those
of flat space as \( t \searrow 0 \).

\subsection{Identifications of the quotients}
\label{sec:ident-quot}

Consider a compact group \( \G \) acting linearly on a
finite-dimensional vector space~\( V \).
A main result of~\cite{Schwarz:smooth-Lie},
cf.~\cite{Mather:invariants}, is that any smooth \( \G \)-invariant
function is necessarily a smooth function of any set of generators for
the ring of \( \G \)-invariant polynomials on~\( V \).
Suppose \( \sigma_1,\dots,\sigma_k \) is a minimal set of such
polynomial generators, meaning that no subset generates.
Then the statement gives that \( \sigma \) induces a diffeomorphism of
\( V/\G \) with \( \sigma(V) \subset \bR^k \) with respect to the
\enquote{smooth structures}: a function on \( V/\G \) is smooth if its
pull-back to \( V \) is smooth; a function on \( \sigma(V) \) is
smooth if it has local extensions to smooth functions in open
\( \bR^k \)-neighbourhoods of each point.

In our cases we are interested in two models:
\begin{enumerate}
\item\label{item:4d} \( \G = S^1 \) acting on \( V = \bR^4 = \bC^2 \)
  as a maximal torus in \( \SU(2) \), and
\item\label{item:6d} \( \G = T^2 \) action on \( V = \bR^6 = \bC^3 \)
  as a maximal torus in \( \SU(3) \).
\end{enumerate}

Let us consider each of these in turn.
For~\ref{item:4d}, let \( (z,w) \) be standard complex coordinates.
Then \( S^1 \) acts as
\( e^{i\theta}(z,w) = (e^{i\theta}z,e^{-i\theta}w) \).
The invariant polynomials are generated by
\( (\sigma_1,\dots,\sigma_4) \):
\begin{gather*}
  \sigma_1 + i\sigma_2 = zw, \quad \sigma_3 = \tfrac12(\abs z^2 - \abs
  w^2), \quad \sigma_4 = \tfrac12(\abs z^2 + \abs w^2).
\end{gather*}
Note that these satisfy the relations
\begin{equation}
  \label{eq:Hopf-graph}
  \sigma_4 \geqslant 0,\qquad
  \sigma_1^2 + \sigma_2^2 + \sigma_3^2 = \sigma_4^2.
\end{equation}

For~\ref{item:6d}, write \( (z_1,z_2,z_3) \) for the standard
coordinates in the flat model, as above.
This time the ring of polynomial invariants is generated by five
elements
\begin{gather*}
  \sigma_1 + i\sigma_2 = - \overline{z_1z_2z_3}, \quad \sigma_3 =
  \tfrac12(\abs{z_2}^2 - \abs{z_3}^2), \quad
  \sigma_4 = \tfrac12(\abs{z_3}^2 - \abs{z_1}^2), \\
  \sigma_5 = \abs{z_3}^2,
\end{gather*}
satisfying the relations
\begin{equation}
  \label{eq:6d-graph}
  \sigma_5 \geqslant \max\Set{0,-2\sigma_3,2\sigma_4}, \qquad
  \sigma_1^2 + \sigma_2^2
  = \sigma_5(\sigma_5 + 2\sigma_3)(\sigma_5 - 2\sigma_4).
\end{equation}

We have chosen our generators in such a way that
\( \sigma_1,\dots,\sigma_{k-1} \) correspond to the relevant
multi-moment maps in the flat models.
Our work in~\S\ref{sec:flat-models} on the flat models shows that in
both cases the map \( \sigma(V) \to \bR^{k-1} \) given by
\( (\sigma_1,\dots,\sigma_{k-1},\sigma_k) \mapsto
(\sigma_1,\dots,\sigma_{k-1}) \) is a homeomorphism.
For the non-flat cases, we have the multi-moment maps giving us
invariant functions that agree with \( \sigma_1,\dots,\sigma_{k-1} \)
to certain orders.
As Schwarz gives that \( V/\G \) is diffeomorphic to \( \sigma(V) \),
the aim is now to show that these still give homeomorphisms to
\( \sigma(V) \to \bR^{k-1} \).
For the case of one-dimensional stabilisers this is
what~\cite{Bielawski:tri-Hamiltonian} does, albeit in a hyperK\"ahler
context, but the local model is the same.
We discuss this briefly as preparation for the six-dimensional case.

For the four-dimensional model we may proceed as follows.
Let \( V \) denote the slice with its \( S^1 \) action.
Write \( \pi \colon V \to V/S^1 \) for the projection.
Use \( W = \bR^4 = U \times \bR \) with \( U = \bR^3 \).
Let \( F_0 \) be the linear projection \( W \to U \).
Let \( S = \sigma(V) \subset W \) be the semi-algebraic set given
by~\eqref{eq:Hopf-graph}.

On the four-dimensional slice~\( V \), we have (restrictions of) the
multi-moment map functions \( \nu_1 \), \( \nu_2 \) and \( \mu \).
Collect these into a single function
\( m = (\nu_1,\nu_2,\mu) \colon V \to \bR^3 \).
This is a smooth invariant function, so by Schwarz it is induced by a
smooth function on~\( S \).  Write
\begin{equation*}
  m = f \circ \sigma, \qquad f\colon S \to \bR^3.
\end{equation*}
Note that \( f \) smooth means it extends to a smooth function in a
neighbourhood of any given point; we use the same name for a choice of
such smooth extension in a neighbourhood of~\( 0 \in W \).

By~\S\ref{sec:near-points-1d}, we know that the first two covariant
derivatives at the origin of \( \nu_1 \), \( \nu_2 \) and \( \mu \)
agree with those of \( \sigma_1 \), \( \sigma_2 \) and \( \sigma_3 \),
respectively.
So \( m \) agrees with \( m_0 = (\sigma_1,\sigma_2,\sigma_3) \) to
order~\( 2 \) near the origin and \( f = F_0 + \tilde f \) with
\( \tilde f \) smooth.
In the slice coordinates at the origin, \( \tilde f \circ \sigma \)
vanishes to order~\( 2 \) and all the \( \sigma_i \) have
degree~\( 2 \), so \( \tilde f \) vanishes to order~\( 1 \)
in~\( \sigma \).  In other words
\begin{equation*}
  \tilde f(\sigma) = \sum_{i,j=1}^{4} \sigma_i\sigma_j f_{ij}(\sigma),
\end{equation*}
where each \( f_{ij} \) is smooth.
In particular, the derivative of \( \tilde f \) has norm bounded above
by \( c\norm\sigma \) on this neighbourhood and the mean value theorem
gives
\begin{equation}
  \label{eq:tfd-quad}
  \norm{\tilde f(x) - \tilde f(y)} \leqslant
  c(\norm{x}+\norm{y})\norm{x-y}.
\end{equation}

Consider points \( q_1 \) and \( q_2 \) in the slice near near the
fixed point \( p = 0 \).
Write \( x = \sigma(q_1) \), \( y = \sigma(q_2) \).  Then
\begin{equation}
  \label{eq:m-lower}
  \begin{split}
    \norm{m(q_1) - m(q_2)}
    &= \norm{f(x) - f(y)}
      = \norm{F_0(x) - F_0(y) + \tilde f(x) - \tilde f(y)} \\
    &\geqslant \norm{F_0(x) - F_0(y)} - c (\norm{x}+\norm{y}) \norm{x-y}.
  \end{split}
\end{equation}
But \( F_0^{-1}(a) = (a,\norm{a}) \in S \) and
\begin{equation*}
  \begin{split}
    \norm{x-y}
    &= \norm[\big]{(F_0(x),\norm{F_0(x)})-(F_0(y),\norm{F_0(y)})} \\
    &\leqslant 2\norm{F_0(x) - F_0(y)}
  \end{split}
\end{equation*}
gives
\begin{equation*}
  \begin{split}
    \norm{m(q_1) - m(q_2)}
    &\geqslant \tfrac12 \norm{x-y} - c (\norm{x}+\norm{y})
      \norm{x-y} \\
    &\geqslant \bigl(\tfrac12 - c(\norm{x}+\norm{y})\bigr)
      \norm{x-y}.
  \end{split}
\end{equation*}
So for \( \norm x, \norm y \leqslant 1/(8c) \), we have
\( \norm{m(q_1) - m(q_2)} \geqslant \norm{x-y}/4 \), proving that
\( m \) is injective on orbits in a neighbourhood of the origin.
Invoking Brouwer's invariance of domain, gives that \( m \) induces a
homeomorphism of the quotient space in a neighbourhood of the origin.

\smallbreak

Let us turn to the six-dimensional models.
Let \( V \) be the slice with its \( T^2 \) action and write
\( \pi \colon V \to V/T^2 \) for the projection map.
Let \( W = \bR^5 = U \times \bR \) with \( U = \bR^4 \) and write
\( F_0\colon W \to U \) for the linear projection.
The vector space \( W \) contains the semi-algebraic set
\( S = \sigma(V) \) given by~\eqref{eq:6d-graph}.
Write \( m = (\nu_1,\mu,\nu_2,\nu_3) \colon V \to \bR^4 \) for the
collection of multi-moment maps.
By Schwarz, \( m = f \circ \sigma \) for a smooth
\( f \colon S \to \bR^4 \).
On \( V \), the first four derivatives of \( \nu_1 \) and \( \mu \),
and the first three derivatives of \( \nu_2 \) and \( \nu_3 \), agree
with those of \( \sigma_1 \), \( \sigma_2 \), \( \sigma_3 \) and
\( \sigma_4 \), respectively.
Noting that any homogeneous polynomial in \( \sigma_i \) of
degree~\( 2 \) is at least of degree \( 4 \) in
the~\( z_i, \overline{z_i} \), we thus have \( f = F_0 + \tilde f \)
with
\begin{equation*}
  \tilde f(\sigma) = \sum_{i,j=1}^5 \sigma_i\sigma_j f_{ij}(\sigma)
\end{equation*}
for some smooth functions \( f_{ij} \).
This gives the estimates~\eqref{eq:tfd-quad} and~\eqref{eq:m-lower} on
some neighbourhood~\( S_0 \) of \( 0 \in S \).

Now consider points \( x = \sigma(q) \)
satisfying~\eqref{eq:6d-graph}.  To estimate \( x_5 \), note that
\begin{equation*}
  x_5(x_5+2x_3)(x_5-2x_4)
  \geqslant
  (x_5 - \max\Set{0,-2x_3,2x_4})^3
\end{equation*}
so, as \( x_5 \geqslant 0 \), we have
\begin{equation*}
  \begin{split}
    \abs{x_5}
    &\leqslant (x_1^2+x_2^2)^{1/3}
      + \max\Set{0,-2x_3,2x_4} \\
    &\leqslant (x_1^2+x_2^2)^{1/3}
      + 2(x_3^2+x_4^2)^{1/2}.
  \end{split}
\end{equation*}
For \( \norm{(x_1,x_2,x_3,x_4)} < 1 \), we have
\begin{equation*}
  \begin{split}
    \norm x
    &= \norm{(F_0(x),x_5)}
      \leqslant \norm{F_0(x)} + \abs{x_5} \\
    &\leqslant \norm{F_0(x)} + \norm{F_0(x)}^{2/3} + 2\norm{F_0(x)} \\
    &\leqslant \norm{F_0(x)}^{2/3}(3\norm{F_0(x)}^{1/3}+1)
      \leqslant 4 \norm{F_0(x)}^{2/3}.
  \end{split}
\end{equation*}
So on \( S_0 \cap B_1(0) \) this gives
\begin{equation*}
  \begin{split}
    \norm{m(q)}
    = \norm{f(x)}
    &\geqslant \norm{F_0(x)} - c \norm x^2 \\
    &\geqslant \Bigl(\tfrac14\norm x\Bigr)^{3/2} - c \norm x^2
      = \norm x^{3/2} \Bigl( \tfrac 18 - c\norm x^{1/2}\Bigr).
  \end{split}
\end{equation*}
Thus for \( \norm x \leqslant 1/(256c^2) \) we have that
\( \norm{m(q)} > \norm x^{3/2}/16 \).
This implies that \( 0 \) is the only point in this
neighbourhood~\( W_0 = \Set{x\in S_0 \cap B_1(0) \with \norm x <
1/(256c^2)} \) that maps to~\( 0 \) under \( m \).

Now consider a family \( \varphi_t \) of \( T^3 \)-invariant
torsion-free \( \G_2 \)-structures on
\( S^1 \times \sigma^{-1}(W_0) \) with \( \varphi_1 = \varphi \), the
structure we are interested in, and \( \varphi_0 \) the flat
\( \G_2 \)-structure that coincides with \( \varphi \) at~\( 0 \).
Such a family was constructed in~\S\ref{sec:deforming-flat-model} and
the discussion there shows that \( f_t \to f_0 = F_0 \) as
\( t \searrow 0 \).
Moreover the bound \( c_t \) above for \( f_t \) also has
\( c_t \searrow 0 \) and in particular \( c = c_1 \geqslant c_t \) for
all \( t<1 \).

Let us consider the Brouwer degrees of these maps,
cf.~\cite{Mawhin:Brouwer-degree,Dinca-M:Brouwer-degree}: let
\( W_1 \subset\subset W_0\) be an open ball containing~\( 0 \); for
\( f \colon W_0 \to \bR^4 \) of class~\( C^2 \) the \emph{Brouwer
degree} is
\begin{equation*}
  d_B[f,W_1] = \int_{W_1} \chi(\norm{f(x)}) J_f(x)\,dx,
\end{equation*}
where \( J_f = \det Df \) is the Jacobian of~\( f \) and
\( \chi\colon \halfopen{0,\infty} \to \halfopen{0,\infty} \) is
continuous, has the closure of its support contained in
\( (0,\inf_{x\in\partial W_1} \norm{f(x)}) \) and satisfies
\( \int_{\bR^4} \chi(\norm x)\,dx = 1 \).
This definition extends to continuous functions~\( f \) by
approximating them uniformly via smooth functions, and the degree is
homotopy invariant; it agrees with the topological degree of the map
\( f/\norm{f} \colon \partial W_1 \to S^3 \).
For \( z \notin f(\partial W_1) \), the Brouwer degree of \( f \)
at~\( z \) is \( d_B[f,W_1,z] = d_B[f(\any)-z,W_1] \).
At regular values~\( z \), the number \( d_B[f,W_1,z] \) counts the
points \( x \)~in \( f^{-1}(z) \cap W_1 \) with the signs
of~\( J_f(x) \).  Any homeomorphism has \( d_B[f,W_1,z] = \pm1 \).

Now \( F_0 = f_0 \) is a homeomorphism \( S \to \bR^4 \) and has
degree~\( +1 \) at all points.
Furthermore, \( S \) is the set set of
\( (\sigma_1,\dots,\sigma_5) \in \bR^5 \)
satisfying~\eqref{eq:6d-graph}.  Differentiating this equation we have
\begin{equation*}
  p_5 d\sigma_5 = \sum_{i=1}^4 p_i d\sigma_i
\end{equation*}
with
\begin{gather*}
  p_1 = 2\sigma_1, \quad p_2 = 2\sigma_2, \quad p_3 =
  -2\sigma_5(\sigma_5 - 2\sigma_4), \quad
  p_4 = 2\sigma_5(\sigma_5 + 2\sigma_3),\\
  \begin{split}
    p_5
    &= (\sigma_5 + 2\sigma_3)(\sigma_5 - 2\sigma_4) +
      \sigma_5(\sigma_5 - 2\sigma_4) + \sigma_5(\sigma_5 + 2\sigma_3) \\
    &= \abs{z_1z_2}^2+\abs{z_3z_1}^2 + \abs{z_2z_3}^2,
  \end{split}
\end{gather*}
where \( (z_1,z_2,z_3) \) are the coordinates on \( V = \bC^3 \).
In particular, \( \sigma_5 \)~is a smooth function of
\( (\sigma_1,\dots,\sigma_4) \) off the locus \( p_5 = 0 \) which is
the image of the set on which two of the \( z_i \) are zero, i.e., the
image of the complex coordinate axes of~\( V \).
But this is just the locus of points with \( T^3 \)-stabiliser of
dimension at least~\( 1 \) and so is specified purely by the group
action.
Off this locus \( dm_t \) has rank~\( 4 \) and so the same is true
of~\( df_t \).
In particular, off this locus \( df_t \) is a local diffeomorphism.
Furthermore, on the locus but away from \( 0 \), we have
\( S^1 \)-stabilisers and from the four-dimensional models we know
that \( f_t \) is a local homeomorphism.

Now homotopy invariance combined with the fact that
\( f_t^{-1}(0) \cap W_1 = \Set{0} \) implies that each \( f_t \) has
degree \( +1 \) and at smooth points the local degrees are
also~\( +1 \).
It follows that on the smooth locus inside \( W_1 \) the maps
\( f_t \) are one-to-one for all \( t \in [0,1] \).
However, the image \( (p_5=0)\setminus\Set{0} \) consists of three
half-lines each determined the group action, in particular by which
copy of \( S^1 \subset T^2 \) is the corresponding stabiliser.
On this set \( m_t \) is still a local homeomorphism and so is
monotone on each half-line.
As \( f_t \) is local homeomorphism it follows that the local degrees
at these points are also \( +1 \).
Thus \( f_t \) is injective on \( W_1 \).
Using Brouwer's invariance of domain, we conclude that \( f \) is a
homeomorphism from \( W_1 \) to a neighbourhood of \( 0 \in \bR^4 \).

Summarising the above analysis, we have shown:

\begin{theorem}
  \label{thm:homeo}
  Let \( (M,\varphi) \) be a toric \( \G_2 \)-manifold.
  Then \( M/T^3 \) is homeomorphic to a smooth four-manifold.
  Moreover, the multi-moment map \( (\nu,\mu) \) induces a local
  homeomorphism \( M/T^3\to\bR^4 \).  \qed
\end{theorem}

\section{Explicit examples of toric
\texorpdfstring{\( \G_2 \)}{G2}-manifolds}

We now turn to write down some explicit examples of toric
\( \G_2 \)-manifolds.

\subsection{Some complete examples}

In this section, we describe some known non-flat complete examples of
toric \( \G_2 \)-manifolds.

\subsubsection{Holonomy \( \SU(3) \): \( M=S^1\times T^*S^3 \)}

Before turning to a concrete example, it seems worthwhile explaining
how it arises as a particular case of a more general construction of
toric \( \G_2 \)-manifolds with holonomy in \( \SU(3) \).
So assume we have a \( 6 \)-manifold \( N \) with vanishing first
Betti number and equipped with a Calabi-Yau structure
\( (\sigma,\Psi) \).
If there is an effective \( T^2 \)-action on \( N \) preserving
\( \sigma \) and~\( \Psi=\psi+i\hat\psi \), then we have invariant
scalar functions \( (\nu,\mu)\colon N\to \bR^4 \) that satisfy the
relations
\begin{equation*}
  \begin{gathered}
    d\nu_1=\psi(U_2,U_3,\any),\quad d\nu_2=-\sigma(U_3,\any),\quad
    d\nu_3=\sigma(U_2,\any),\\
    d\mu=-\hat\psi(U_2,U_3,\any),
  \end{gathered}
\end{equation*}
where \( U_2,U_3 \) are generators for the torus action.
We can now consider the torsion-free product \( \G_2 \)-structure on
\( M= S^1\times N \) given by
\begin{equation*}
  \varphi=dx\wedge\sigma+\psi,\quad
  \Hodge\varphi=\hat\psi\wedge dx+\tfrac12\sigma^2.
\end{equation*}
Clearly, \( (M,\varphi) \) is toric with \( T^3=S^1\times T^2 \)
acting in the obvious way and associated multi-moment maps
\( (\nu,\mu) \).
Theorem~\ref{thm:homeo} now implies that \( N/T^{2} \) is locally homeomorphic
to~\( \bR^{4} \) and Lemma~\ref{lem:sing-point} implies that the
trivalent graphs lie in the surfaces \( (\nu_{1},\mu) \) constant.

For \( (N,\sigma,\Psi) \) as above there is a special Lagrangian
foliation (of an open dense subset) with \( T^2 \)-symmetry.
The leaves are given by fixing \( (\nu_2,\nu_3,\mu) \) to be constant.
The corresponding distribution is given by the kernel of
\( d\mu\wedge d\nu_{23} \), and the restriction of \( \psi \) to each
leaf is \( \theta_{23}\wedge d\nu_1 \).

As a concrete example of the above, one can take \( N=T^*S^3 \) with
its Stenzel Calabi-Yau structure~\cite{Stenzel:Ricci-flat}.
For our purposes, it is more convenient to identify \( N \) with the
complex sphere
\begin{equation*}
  Q=\Set[\Big]{z\in \bC^4 \with \sum_{j=0}^3z^2_j=1},
\end{equation*}
following~\cite{Ionel-M:SLag-T2}.
Specifically, one has the \( \SO(4) \)-equivariant diffeomorphism
\begin{equation*}
  T^*S^3\ni(p,v)
  \mapsto \cosh(\norm{v})p+i\sinh(\norm{v})\frac{v}{\norm{v}}\in Q,
\end{equation*}
(see~\cite{Szoke:cotangent-diffeo}).
In terms of \( Q \), the K\"ahler \( 2 \)-form is given by
\( \sigma=d\alpha \), where
\begin{equation*}
  \alpha(X)_z=\tfrac12f'(\abs{z}^2)\im(X^t\overline{z}),
  \qquad X\in T_z Q,\; z\in Q,
\end{equation*}
with \( f \) satisfying the following differential equation:
\begin{equation*}
  ((f_u)^3)_u=3k(\sinh u)^2,
\end{equation*}
for some constant \( k>0 \).
The holomorphic volume form can be computed as
\begin{equation*}
  \Psi(X_1,X_2,X_3)_z = dz_{0123}(z,X_1,X_2,X_3),
\end{equation*}
for \( X_1,X_2,X_3\in T_z Q \) and \( z\in Q \).

For the \( T^2 \)-action, we consider \( T^2\subset \SO(4) \)
generated by the vector fields
\begin{equation*}
  U_2(z)=(-z_1,z_0,0,0),\quad U_3(z)=(0,0,-z_3,z_2).
\end{equation*}
In accordance with~\cite[Thm.~5.2]{Ionel-M:SLag-T2} one finds that the
multi-moment maps are
\begin{gather*}
  \nu_1+i\mu=\tfrac12(\bar{z}_0^2+\bar{z}_1^2),\quad
  \nu_2=-f'(\abs{z}^2)\im(z_2\bar{z}_3),\quad
  \nu_3=f'(\abs{z}^2)\im(z_0\bar{z}_1).
\end{gather*}

Many other examples are to be found in
\cite{Aganagic-KMV:top-vertex,Li-LLZ:top-vertex} and related works.

\subsubsection{The cone over \( S^3\times S^3 \) and its deformation}
\label{sec:Cone-and-BS-S3S3}

As mentioned in~\S\ref{sec:eff-T3}, one example of a complete toric
\( \G_2 \)-manifold with holonomy equal to \( \G_2 \) is the spin
bundle over \( S^3 \) equipped with its Bryant-Salamon structure.
It may be viewed as a deformation of the cone over \( S^3\times S^3 \)
with its nearly K\"ahler structure.
In both cases, one can describe the \( \G_2 \)-structure in terms of
one-parameter families of left-invariant half-flat
\( \SU(3) \)-structures on
\( S^3\times S^3\cong\Sp(1)\times\Sp(1)\subset \bH\times \bH \).

To make this concrete, let us take
\( \Set{(i,0),(j,0),(-k,0),(0,i),(0,j),(0,-k)} \) as our basis of
\( \lie{sp}(1)\oplus\lie{sp}(1)\cong T_1(S^3\times S^3) \).
Correspondingly, the tangent space at \( (p,q)\in S^3\times S^3 \) has
basis
\begin{equation}
  \label{eq:nK-basis}
  \begin{gathered}
    E_1(p,q)=(pi,0),\quad E_2(p,q)=(pj,0),\quad E_3(p,q)=(-pk,0),\\
    F_1(p,q)=(0,qi),\quad F_2(p,q)=(0,qj),\quad F_3(p,q)=(0,-qk).
  \end{gathered}
\end{equation}
If we let \( e^1,\dots, f^3 \) denote the dual co-frame, then
\( de^i=2e^{jk} \) and \( df^i=2f^{jk} \), \( (ijk)=(123) \).

We have an almost effective action of \( \Sp(1)^3 \) on
\( S^3\times S^3 \) given by
\begin{equation*}
  ((h,k,\ell),(p,q))\mapsto(hp\ell^{-1},kq\ell^{-1})
\end{equation*}
that preserves the half-flat \( \SU(3) \)-structures of interest
(cf.~\cite{Conti-M:SO3}).
By choosing a maximal torus \( S^1 \) in each \( \Sp(1) \), we obtain
an almost effective action of~\( T^3 \).
Considering the quotient of \( T^3 \) by
\( \mathbb Z_2=\Set{\pm(1,1,1)} \), we get an effective action of a
torus~\( T^3 \).
For concreteness, let us choose each maximal torus
\( T^1\subset\Sp(1) \) to be of the form
\( \Set{e^{i\theta} \with \theta\in\bR} \).
In this case, we have generating vector fields given by
\begin{equation*}
  U_1(p,q)=(ip,0),\quad U_2(p,q)=(0,iq),\quad U_3(p,q)=(-pi,-qi).
\end{equation*}
Following~\cite{Dixon:multi-moment-im}, we can express these vector
fields in terms of~\eqref{eq:nK-basis} via
\begin{equation*}
  \begin{gathered}
    U_1(p,q)=\inp{\bar{p}ip}{i}E_1(p,q)+\inp{\bar{p}ip}{j}E_2(p,q)-\inp{\bar{p}ip}{k}E_3(p,q),\\
    U_2(p,q)=\inp{\bar{q}iq}{i}F_1(p,q)+\inp{\bar{q}iq}{j}F_2(p,q)-\inp{\bar{q}iq}{k}F_3(p,q),\\
    U_3(p,q)=-E_1(p,q)-F_1(p,q),
  \end{gathered}
\end{equation*}
where \( \inp{\any}{\any} \) is the usual inner product on
\( \im\bH\cong\bR^3 \).
Note that each of the maps \( p\mapsto \bar{p}ip \),
\( q\mapsto\bar{q}iq \) is a standard Hopf fibration
\( \pi_{H}\colon S^3\to S^2\subset\im\bH \).
We see that the span of the \( U_1,U_2,U_3 \) is \( 3 \)-dimensional,
unless
\( p,q \in\pi_{H}^{-1}(\Set{\pm i}) = \Set{e^{i\theta},je^{i\theta}
\with \theta\in\bR} \).

The nearly K\"ahler structure on \( S^3\times S^3 \) can be expressed
as
\begin{equation*}
  \begin{gathered}
    \sigma=\tfrac{2}{3\sqrt3}(e^1f^1+e^2f^2+e^3f^3),\\
    \psi=\tfrac4{9\sqrt{3}}(e^{23}f^1+e^{31}f^2+e^{12}f^3-e^1f^{23}-e^2f^{31}-e^3f^{12}),\\
    \hat{\psi}=\tfrac4{27}(-2e^{123}-2f^{123}+e^1f^{23}+e^2f^{31}+e^3f^{12}+e^{23}f^1+e^{31}f^2+e^{12}f^3).
  \end{gathered}
\end{equation*}
Specifically this means that \( (\sigma,\psi) \) defines an
\( \SU(3) \)-structure satisfying \( d\sigma=3\psi \) and
\( d\hat\psi=-2\sigma^2 \).
As mentioned above, \( T^3 \)~acts effectively, preserving the nearly
K\"ahler structure, and we have associated multi-moment maps
\( (\tilde\nu,\tilde\mu)\colon S^3\times S^3\to\bR^4 \) for the pair
of closed forms \( (\psi,\sigma^2) \).
As \( d\sigma=3\psi \) and \( d\hat{\psi}=-2\sigma^2 \), it is
particularly easy to compute the maps \( (\tilde\nu,\tilde\mu) \):
by~\cite[Prop.~3.1]{Madsen-S:mmmap2} we have that
\( \tilde{\nu}_i=\tfrac13\sigma(U_j,U_k) \) and
\( \tilde{\mu}=\tfrac12\hat\psi(U_1,U_2,U_3) \).

The conical \( \G_2 \)-structure on \( \bR_+\times S^3\times S^3 \) is
given by
\begin{equation*}
  \begin{gathered}
    \varphi_C=dr\wedge r^2\sigma+r^3\psi=d(\tfrac13r^3\sigma),\quad
    \Hodge{\varphi}_C=r^3\hat{\psi}\wedge
    dr+\tfrac12r^4\sigma^2=d(-\tfrac14r^4\hat{\psi}).
  \end{gathered}
\end{equation*}
It follows that
\begin{equation*}
  U_i\wedge U_j\hook\varphi
  = 3r^2\tilde{\nu}_k dr + r^3d\tilde{\nu}_k
  = d(r^3\tilde{\nu}_k)
\end{equation*}
and
\begin{equation*}
  U_1\wedge U_2\wedge U_3\hook\Hodge\varphi=2r^3\tilde\mu dr+\tfrac12r^4d\tilde\mu=d(\tfrac12r^4\tilde\mu).
\end{equation*}
So in terms of nearly K\"ahler data, the multi-moment maps
\( (\nu^C,\mu^C)\colon \bR_+\times S^3\times S^3\to\bR^4 \) are given
by \( (\nu^C,\mu^C)=(r^3\tilde{\nu},\frac{r^4}2\tilde\mu) \).
Explicitly,
\begin{equation*}
  \begin{gathered}
    \nu^C_1(r,(p,q))=\tfrac{2r^3}{9\sqrt3}\inp{\bar{q}iq}{i},\quad \nu^C_2(r,(p,q))=\tfrac{2r^3}{9\sqrt3}\inp{\bar{p}ip}{i},\\
    \nu^C_3(r,(p,q))=\tfrac{2r^3}{9\sqrt3}\inp{\bar{p}ip}{\bar{q}iq},\\
    \mu^C(r,(p,q))=\tfrac{2r^4}{27}\bigl(\inp{\bar{p}ip}{j}\inp{\bar{q}iq}{k}-\inp{\bar{p}ip}{k}\inp{\bar{q}iq}{j}\bigr).
  \end{gathered}
\end{equation*}

From the remarks about Hopf-fibrations, it is clear that
\( (\nu^C,\mu^C) \) induces a map
\( \bR_+\times S^2\times S^2\to \bR^4 \) given by
\begin{equation*}
  \begin{gathered}
    (r,(v,w)) \mapsto \tfrac{2r^3}{9\sqrt3} \bigl(\inp{v}{i},
    \inp{w}{i}, \inp{v}{w},
    \tfrac{2r}{\sqrt3}(\inp{v}{j}\inp{w}{k}-\inp{v}{k}\inp{w}{j})\bigr).
  \end{gathered}
\end{equation*}

Turning now to the Bryant-Salamon solution on the spin bundle of
\( S^3 \), we begin by observing that this can be written in the form
\begin{equation*}
  \begin{gathered}
    \varphi_{BS}
    =-\tfrac4{3\sqrt3}\epsilon(e^{123}-f^{123})+d(\tfrac13(r^3-\epsilon)\sigma),\\
    \Hodge{\varphi_{BS}} =\tfrac49\epsilon dr\wedge(e^{123}+f^{123})
    +(r^3-\epsilon)\hat\psi\wedge dr+\tfrac12
    r(r^3-4\epsilon)\sigma^2,
  \end{gathered}
\end{equation*}
for some \( \epsilon>0 \) (see, e.g., \cite{Brandhuber-al:G2}).
Then, building on the computations from the nearly K\"ahler case, we
find that the multi-moment maps for the toric Bryant-Salamon manifold
are
\begin{equation*}
  \begin{gathered}
    \nu^{BS}_1(r,(p,q))=\tfrac2{9\sqrt3}(r^3-4\epsilon)\inp{\bar{q}iq}{i},\\
    \nu^{BS}_2(r,(p,q))=\tfrac2{9\sqrt3}(r^3-4\epsilon)\inp{\bar{p}ip}{i},\\
    \nu^{BS}_3(r,(p,q))=\tfrac2{9\sqrt3}(r^3-\epsilon)\inp{\bar{p}ip}{\bar{q}iq},\\
    \mu^{BS}(r,(p,q))=\tfrac2{27}r(r^3-4\epsilon)\bigl(\inp{\bar{p}ip}{j}\inp{\bar{q}iq}{k}-\inp{\bar{p}ip}{k}\inp{\bar{q}iq}{j}\bigr).
  \end{gathered}
\end{equation*}
In this case, the matrix \( V \) has inverse given by
\begin{equation*}
  V^{-1}
  =\begin{pmatrix}
    \tfrac{4(r^3-\epsilon)}{9r} & -\tfrac{\sqrt3}r\tfrac{2\epsilon+r^3}{r^3-\epsilon}\nu_3^{BS}& -\tfrac{\sqrt3}r\nu_2^{BS}\\
    -\tfrac{\sqrt3}r\tfrac{2\epsilon+r^3}{r^3-\epsilon}\nu_3^{BS} & \tfrac{4(r^3-\epsilon)}{9r} & -\tfrac{\sqrt3}r\nu_1^{BS}\\
    -\tfrac{\sqrt3}r\nu_2^{BS} & -\tfrac{\sqrt3}r\nu_1^{BS} &\tfrac{4(r^3-4\epsilon)}{9r}
  \end{pmatrix}.
\end{equation*}

We obtain the values of the multi-moment map on the zero section of
the spin bundle by continuity.
Away from this zero section, the points with one-dimensional
stabilisers map to the straight lines
\( (\varepsilon_1 t, \varepsilon_2 t, \varepsilon_1\varepsilon_2
(t+k), 0) \), where \( \varepsilon_i \in \Set{\pm1} \),
\( k = 2\epsilon/(3\sqrt3) \) and \( t > 0 \).
The limit \( t\searrow 0 \) gives points with stabiliser \( T^2 \) and
the preimages of the interior of the line segment from
\( (0,0,-k,0) \) to \( (0,0,k,0) \) is also a family of points with
one-dimensional stabiliser.
The image of the singular orbits is thus of the form
\scalebox{.5}{\begin{picture}(45,20)(0,0) \put(0,0){\line(1,1){10}}
  \put(0,20){\line(1,-1){10}} \put(10,10){\line(1,0){20}}
  \put(30,10){\line(1,-1){10}} \put(30,10){\line(1,1){10}}
\end{picture}}.

For \( r \) fixed large, \( (\nu, \mu/r) \) essentially induces the
map
\( (x, z, y, w) \mapsto (x,y,xy + \norm z\norm w \cos\theta, \norm
z\norm w \sin\theta) \), where
\( (x,z),(y,w) \in S^2 \subset \bR \times \bC \) and \( \theta \)~is
the oriented angle from \( z \) to~\( w \).
On the quotient space this map is thus a homeomorphisms of topological
three spheres and of global degree~\( 1 \).
From the general theory, we know \( (\nu,\mu) \) has local degree
\( +1 \), so we conclude that the multi-moment map is injective on the
orbit space.
However, varying the parameter \( r \), we get a deformation retract
to the ellipsoids to the line segment
\( \Set{(0,0,t,0) \with t \in [-k,k] } \), so the multi-moment map is
onto.
We conclude that the multi-moment map is a homeomorphism from the
\( T^3 \) orbit space of the spin bundle onto \( \bR^4 \).

\begin{remark}
  After completing this paper, Foscolo, Haskins and Nordst\"om
  \cite{Foscolo-HN:G2TNEH} have constructed many new examples of
  \( \G_{2} \)-manifolds, including several examples with
  \( T^{3} \)-symmetry.
  For some of these, we find that the corresponding trivalent graphs
  are planar (see~\cite{Swann:Waterloo}), even though the holonomy
  group is the whole of~\( \G_{2} \).
\end{remark}

\subsection{Ans\"atze simplifying the PDEs}
\label{sec:cases}

From a PDE viewpoint a particular challenge is the fact that the
characterisation of toric \( \G_2 \)-manifolds involves the coupled
system consisting of both first order PDEs~\eqref{eq:div-free} and a
second order system~\eqref{eq:elliptic}.
In the following, we shall study some special cases that circumvent
this complicating issue.
This allows us to construct many explicit (but generally incomplete)
examples of toric \( \G_2 \)-manifolds.
In particular, we find that simple polynomial solutions in the
variables \( (\nu,\mu) \) can lead to metrics with holonomy equal to
\( \G_2 \).

\subsubsection{One variable dependence}

Let us assume that \( V \) depends only on the variable \( \mu \), so
\( \partial V/\partial \nu_i=0 \), \(i=1,2,3 \).
Then \( Z\equiv0 \).
The condition that \( d\omega=0 \) now yields that
\( \partial^2 V_{ij}/\partial\mu^2=0 \).
So \( V \) is linear in \( \mu \) and thus \( W \) is constant.

\begin{example}
  \label{ex:mu-dep}
  Taking \( V=\diag(\mu,\mu,\mu) \) gives a solution defined for all
  \( \mu>0 \).
  In this case, the associated \( \G_2 \)-metric takes the form
  \begin{equation*}
    g=\tfrac1{\mu}(\theta_1^2+\theta_2^2+\theta_3^2)
    + \mu^2(d\nu_1^2+d\nu_2^2+d\nu_3^2)+\mu^3d\mu^2,
  \end{equation*}
  where \( d\theta_i=d\nu_j\wedge d\nu_k \), \( (ijk)=(123) \).

  This metric has (restricted) holonomy equal to \( \G_2 \) as can be
  seen, e.g., by computing the Riemannian curvature: regarded as a
  \( 2 \)-form \( \Omega=(\Omega_{ij}) \) on \( T^3\times \cU \) with
  values in an associated \( \lie g_2 \)-bundle, the span of
  \( \Omega_{ij} \), \( 1\leqslant i\leqslant j\leqslant 7 \), has
  dimension \( 14 \).
\end{example}

From the viewpoint of complete metrics, this situation turns out to be
less interesting.

\begin{proposition}
  \label{prop:incompl-mu-dep}
  Suppose \( V = V(\mu) \).
  If \( (M,\varphi) \) is complete, then it is flat and hence locally
  isometric to \( \bR^7 \).
\end{proposition}

\begin{proof}
  By Corollary~\ref{cor:V-const}, it suffices to show that
  completeness forces \( V \) to be a constant matrix.
  So let us assume \( V \) is not constant.

  After adding a constant to \( \mu \), if necessary, we may assume
  that \( V(0)> 0 \) and then it follows by
  Remark~\ref{rem:GL3act-quadform} that we can take \( V(0)=1_3 \).
  In fact, using the action of \( \GL(3,\bR) \) on \( S^2(\bR^3) \),
  we can even assume \( V \) has the form
  \( V(\mu)=\diag( \lambda_1\mu+1,\lambda_2\mu+1,\lambda_3\mu+1) \)
  where \( \lambda_1\geqslant\lambda_2\geqslant\lambda_3 \).

  As \( V \) is not constant, there is \( \lambda_i\neq0 \) such that
  the rank of \( V \) drops (the first time) when
  \( \mu=-1/\lambda_i \).
  By Lemma~\ref{lem:sing-point}, we cannot be approaching a point
  \( p\in M\setminus M_0 \), i.e, a singular orbit, as we have
  \( \det(B)\to\infty \).
  To show that this implies incompleteness of the \( \G_2 \)-metric,
  we use the criterion of~\cite[Lem.~1]{Cortes-al:completeness}: we
  look for a finite length curve not contained in any compact set.

  In the base space of our \( T^3 \)-bundle, we have a curve
  \( \gamma \), defined on \( \halfclosed{-1/\lambda_i,0} \),
  corresponding to a curve parameterised by the \( \mu \)-coordinate.
  Let \( p\in M_0 \) be a point projecting to \( \gamma(0) \) and
  \( \tilde\gamma \) the horizontal lift of \( \gamma \) with
  \( \tilde\gamma(0)=p \).
  Clearly, the curve
  \( \tilde\gamma\colon\halfclosed{-1/\lambda_i,0}\to M_0 \) has
  finite length, but is not contained in any compact set.
\end{proof}

In the cases where \( V \) depends only on one of the variables
\( \nu_i \), similar arguments and conclusions apply.

\subsubsection{Orthogonal Killing vectors}

Let us assume \( V_{ij}=0 \) for all \( i\neq j \), i.e., the
generating vector fields for the torus action are orthogonal.
The \( \G_2 \)-metric now takes the form
\begin{equation*}
  g = \tfrac1{V_{11}}\theta_1^2 + \tfrac1{V_{22}}\theta_2^2 +
  \tfrac1{V_{33}}\theta_3^2
  + V_{11}V_{22}V_{33} \bigl(d\mu^2 + \tfrac1{V_{11}}d\nu_1^2 +
  \tfrac1{V_{22}}d\nu_2^2 + \tfrac1{V_{33}}d\nu_3^2 \bigr).
\end{equation*}

In this case, \( W \) is diagonal with non-zero entries given by
\( w^j_j=\partial V_{jj}/\partial\mu \), and \( Z \) has zeros on the
diagonal and off-diagonal entries given by
\begin{equation*}
  z_i^j=-V_{kk}\D{V_{ii}}{\nu_k},\quad z_j^i=V_{kk}\D{V_{jj}}{\nu_k},
\end{equation*}
with \( (ijk)=(123) \).

The divergence-free condition~\eqref{eq:div-free} tells us that
\( \partial V_{ii}/\partial\nu_i=0 \), for \( i=1,2,3 \).
Then the condition \( d\omega=0 \) is given by the equations
\begin{equation}
\label{eq:diagoonal_2nd_order}
  \Dsq{V_{ii}}{\mu}+ V_{jj}\Dsq{V_{ii}}{\nu_j} + V_{kk}\Dsq{V_{ii}}{\nu_k}=0 \qquad (ijk)=(123)
\end{equation}
together with
\begin{equation}
\label{eq:diagoonal_1st_order}
  \D{V_{ii}}{\nu_j}\D{V_{jj}}{\nu_i}=0
\end{equation}
for \( i\neq j \).

Assume now that one has \( \partial V_{ii}/\partial\nu_j\neq0 \), for some
\( j\neq i \).
Without loss of generality, we can take
\( \partial V_{11}/\partial\nu_2\neq0 \), which forces
\( \partial V_{22}/\partial\nu_1=0 \).
So \( V_{22} \) is a function of \( \nu_3 \) and \( \mu \) alone.
By differentiating the equation \eqref{eq:diagoonal_2nd_order}, for \( i=2 \),
we then find that
\begin{equation*}
  \D{V_{33}}{\nu_1}\Dsq{V_{22}}{\nu_3}=0=\D{V_{33}}{\nu_2}\Dsq{V_{22}}{\nu_3}.
\end{equation*}
So either \( \partial^2V_{22}/\partial\nu_3^2 \) vanishes identically,
or there is an open set where
\( \partial V_{33}/\partial \nu_i=0 \), for \( i=1,2,3 \).

In the first case, \( V_{22} \), as a function of \( \nu_3 \), has
non-vanishing derivative of order at most one and so is either
constant or linear in that variable. Correspondingly, we have
\( \partial V_{22}/\partial\nu_3=0 \) or
 \( \partial V_{22}/\partial\nu_3\neq0 \), respectively.

If \( \partial V_{22}/\partial\nu_i=0 \), \( i=1,2,3 \), the additional
information captured by \eqref{eq:diagoonal_1st_order}
is that either \( \partial V_{11}/\partial\nu_3=0 \) or
\( \partial V_{33}/\partial\nu_1=0 \) in an open neighbourhood.
If \( \partial V_{22}/\partial\nu_3\neq0 \), then \eqref{eq:diagoonal_1st_order} moreover
tells us that  \( \partial V_{33}/\partial\nu_2=0 \).

Considering the case where \( \partial V_{22}/\partial\nu_3\neq0 \) and
\begin{equation*}
  \D{V_{11}}{\nu_3}=0=\D{V_{22}}{\nu_1}=\D{V_{33}}{\nu_2},
\end{equation*}
 \eqref{eq:diagoonal_2nd_order} reduces to the equations
\begin{equation*}
  \Dsq{V_{11}}{\mu}+ V_{22}\Dsq{V_{11}}{\nu_2}=0,\quad
  \Dsq{V_{33}}{\mu}+ V_{11}\Dsq{V_{33}}{\nu_1}=0.
\end{equation*}
Differentiating the first of these expressions with respect to
\( \nu_3 \), we find that \( V_{11} \) is (at most) linear in
\( \nu_2 \). Similarly, from differentiating the second equation above with respect to
\( \nu_2 \), we find that either \( V_{33} \) is (at most) linear in
\( \nu_1 \) otherwise \( \partial V_{11}/\partial\nu_2=0 \) on an open
set.

After possibly relabelling indices, the above considerations imply that there are two ways
to satisfy \eqref{eq:diagoonal_2nd_order} and \eqref{eq:diagoonal_1st_order}.
The first one is to have each \( V_{ii} \) (at most) a linear function
in two variables as follows:
\begin{equation}
  \label{eq:diag-ans-casei}
  V_{11}= V_{11}(\mu,\nu_2),\quad V_{22}= V_{22}(\mu,\nu_3),\quad
  V_{33}= V_{33}(\mu,\nu_1).
\end{equation}
From the viewpoint of complete metrics this is less interesting:

\begin{proposition}
  If \( (M,\varphi) \) is complete with \( V \) diagonal and its
  entries satisfy~\eqref{eq:diag-ans-casei}, then \( (M,\varphi) \) is
  flat and hence locally isometric to \( (\bR^7,\varphi_0) \).
\end{proposition}

\begin{proof}
  This is essentially proved in the same way as
  Proposition~\ref{prop:incompl-mu-dep}.
  We may assume that \( V(0)>0 \).
  Consequently, we can write \( V \) in the form:
  \begin{equation*}
    \diag(
    \epsilon_1\nu_2\mu +\kappa_1\nu_2+\lambda_1+1,
    \epsilon_2\nu_3\mu+\kappa_2\nu_3+ \lambda_2+1,
    \epsilon_3\nu_1\mu + \kappa_3\nu_1+\lambda_3+1 ).
  \end{equation*}
  By considering suitable curves (corresponding to \( (0,\mu) \),
  \( (\nu_1,0) \) etc.), we arrive at the asserted conclusion.
\end{proof}

The second and more interesting possibility is to have
\(  \partial V_{33}/\partial \nu_i=0 \), \( i=1,2,3 \), together with
\begin{equation*}
  \D{V_{22}}{\nu_1}=0=\D{V_{22}}{\nu_2},\quad
  \D{V_{11}}{\nu_1}=0.
\end{equation*}
In this case, \eqref {eq:diagoonal_2nd_order} corresponds to the following elliptic hierarchy:
\begin{equation}
  \label{eq:PDE-diag}
  \Dsq{V_{11}}{\mu}+ V_{22}\Dsq{V_{11}}{\nu_2} +
  V_{33}\Dsq{V_{11}}{\nu_3}=0, \quad \Dsq{V_{22}}{\mu} +
  V_{33}\Dsq{V_{22}}{\nu_3}=0,\quad \Dsq{V_{33}}{\mu}=0.
\end{equation}
So again \( V_{33} \) is at most a linear function of \( \mu\), and
\( V \) is independent of \( \nu_1 \).
This means, in particular, that \( U_2 \) and \( U_3 \) have no zeros,
i.e., there are no points with \( T^2 \)-isotropy, and points with
\( S^1 \)-isotropy lie above disjoint lines parallel to the \( \nu_{1} \)-axis.

When \( V_{33} \) is constant, which we can take to be \( 1 \), the
\( \G_2 \)-metric is a product
\begin{equation*}
  g = \theta_3^2 + \tfrac1{V_{11}}\theta_1^2 +
  \tfrac1{V_{22}}\theta_2^2 + V_{11}V_{22} \bigl(d\mu^2 +
  \tfrac1{V_{11}}d\nu_1^2 + \tfrac1{V_{22}}d\nu_2^2 + d\nu_3^2 \bigr),
\end{equation*}
so the holonomy reduces to a subgroup of \( \SU(3) \).

Reducing the holonomy further, one obvious solution to the elliptic
system in this case is given by taking \( V_{22}=1=V_{33} \) and
\( V=V_{11}(\mu,\nu_2,\nu_3) \) to be a harmonic function on an
\( \bR^3 \).
Then the associated \( \G_2 \)-holonomy metric is given by
\begin{equation*}
  g = \theta_2^2 + \theta_3^2 + d\nu_1^2
  + \tfrac1{V} \theta_1^2 + V\,(d\mu^2+d\nu_2^2+d\nu_3^2).
\end{equation*}
This has the form of a product of a flat metric on (an open set of)
\( T^2\times \bR \) and a hyperK\"ahler metric on an \( S^1 \)-bundle
over (an open set of) \( \bR^3 \).

Excluding these cases of reduced holonomy, we are thus left with
analysing the equations:
\begin{equation*}
  \begin{gathered}
    \Dsq{V_{11}}{\mu}+ V_{22}\Dsq{V_{11}}{\nu_2} +
    \mu\Dsq{V_{11}}{\nu_3}=0, \quad \Dsq{V_{22}}{\mu} +
    \mu\Dsq{V_{22}}{\nu_3}=0,
  \end{gathered}
\end{equation*}
having set \( V_{33}(\mu)=\mu \).

As the following example shows, it is easy to find local (incomplete)
solutions to these equations that have full holonomy.

\begin{example}
  \label{ex:poly-ex}
  By writing down \( (\nu,\mu) \) as a power series and
  solving~\eqref{eq:PDE-diag}, we get solutions on trivial bundles
  \( T^3\times \cU \), where \( \cU\subset\bR^4 \) is an appropriate
  open subset.  As an example of such a solution we can take
  \begin{equation*}
    V_{11}(\nu_2,\nu_3,\mu)=2\mu^5-15\mu^2\nu_3^2-5\nu_2^2,\quad
    V_{22}(\nu_3,\mu)=\mu^3-3\nu_3^2,\quad
    V_{33}(\mu)=\mu.
  \end{equation*}
  As in Example~\ref{ex:mu-dep}, one checks by explicit computations
  that the associated metric has (restricted) holonomy equal
  to~\( \G_2 \).
\end{example}

\providecommand{\bysame}{\leavevmode\hbox to3em{\hrulefill}\thinspace}
\providecommand{\MR}{\relax\ifhmode\unskip\space\fi MR }
\providecommand{\MRhref}[2]{%
  \href{http://www.ams.org/mathscinet-getitem?mr=#1}{#2}
}
\providecommand{\href}[2]{#2}

\end{document}